\documentclass[a4paper]{amsproc}
\usepackage{amssymb}
\usepackage{mathrsfs}
\usepackage{amsfonts}
\usepackage{amsthm}
\usepackage{amsmath,amscd}
\usepackage{amsxtra}     
\usepackage{bm}
\usepackage[all,cmtip]{xy}
\usepackage{epsfig}
\usepackage{verbatim}
\usepackage{color}
\usepackage{enumerate}
\usepackage{hyperref}
\usepackage{tikz}
\usetikzlibrary{arrows,matrix,shapes,trees}
\usepackage{extarrows}

\theoremstyle{plain}
 \newtheorem{thm}{Theorem}[section]
 \newtheorem{prop}[thm]{Proposition}
 \newtheorem{lem}[thm]{Lemma}
 \newtheorem{cor}[thm]{Corollary}

 \newtheorem{lem'}[thm]{``Lemma''}
 
\theoremstyle{definition}

\theoremstyle{remark}
 \newtheorem{rmk}{Remark}[section]
 \numberwithin{equation}{section}

\newcommand{\N}{{\mathbb N}}
\newcommand{\Q}{{\mathbb Q}}
\newcommand{\Z}{{\mathbb Z}}

\newcommand{\F}{{\mathbb F}}
\newcommand{\Gm}{\mathbb{G}_{\mr{m}}}

\newcommand{\Gml}{\mathbb{G}_{\mr{m,log}}}
\newcommand{\Gmlb}{\overline{\mathbb{G}}_{\mr{m,log}}}

\newcommand{\mr}{\mathrm}
\newcommand{\mc}{\mathcal}

\newcommand{\fsX}{(\mr{fs}/X)}
\newcommand{\fsXet}{(\mr{fs}/X)_{\mr{\acute{e}t}}}
\newcommand{\fsXfl}{(\mr{fs}/X)_{\mr{fl}}}
\newcommand{\fsXket}{(\mr{fs}/X)_{\mr{k\acute{e}t}}}
\newcommand{\fsXkfl}{(\mr{fs}/X)_{\mr{kfl}}}
\newcommand{\Xet}{X_{\mr{\acute{e}t}}}
\newcommand{\Xket}{X_{\mr{k\acute{e}t}}}
\newcommand{\Xfl}{X_{\mr{fl}}}
\newcommand{\Xkfl}{X_{\mr{kfl}}}
\newcommand{\Spec}{\mathop{\mr{Spec}}}

\title[Comparison of Kummer topologies with classical topologies]{Comparison of Kummer logarithmic topologies with classical topologies}



\author[Heer Zhao]{\bfseries Heer Zhao}
\address{Fakult\"at f\"ur Mathematik, 
    Universit\"at Duisburg-Essen, 
    Essen 45117, 
    Germany}
\email{heer.zhao@uni-due.de}

\begin{document}

\vspace{18mm} \setcounter{page}{1} \thispagestyle{empty}

\begin{abstract}
We compare the Kummer flat (resp. Kummer \'etale) cohomology with the flat (resp. \'etale) cohomology with coefficients in smooth commutative group schemes, finite flat group schemes and Kato's logarithmic multiplicative group. We will be particularly interested in the case of algebraic tori in the Kummer flat topology. We also make some computations for certain special cases of the base log scheme.
\end{abstract}

\maketitle

\section*{Notation and conventions}
Let $X=(X,M_X)$ be an fs log scheme, we denote by $(\mr{fs}/X)$ the category of fs log schemes over $X$, and denote by $(\mr{fs}/X)_{\mr{\acute{e}t}}$ (resp. $(\mr{fs}/X)_{\mr{fl}}$, resp. $(\mr{fs}/X)_{\mr{k\acute{e}t}}$, resp. $(\mr{fs}/X)_{\mr{kfl}}$) the classical \'etale site (resp. classical flat site, resp. Kummer \'etale site, resp. Kummer flat site) on $(\mr{fs}/X)$. In order to shorten formulas, we will mostly abbreviate $\fsXet$ (resp. $\fsXfl$, resp. $\fsXket$, resp. $\fsXkfl$) as $\Xet$ (resp. $\Xfl$, resp. $\Xket$, resp. $\Xkfl$). We refer to \cite[2.5]{ill1} for the classical \'etale site and the Kummer \'etale site, and \cite[Def. 2.3]{kat2} and \cite[\S 2.1]{niz1} for the Kummer flat site. The site $\fsXfl$ is an obvious analogue of $\fsXet$. We have natural ``forgetful'' maps of sites:
\begin{equation}\label{eq0.1}
\varepsilon_{\mr{\acute{e}t}}:(\mr{fs}/X)_{\mr{k\acute{e}t}}\rightarrow (\mr{fs}/X)_{\mr{\acute{e}t}}  
\end{equation}
and
\begin{equation}\label{eq0.2}
\varepsilon_{\mr{fl}}:(\mr{fs}/X)_{\mr{kfl}}\rightarrow (\mr{fs}/X)_{\mr{fl}}.
\end{equation}
We denote by $(\mr{st}/X)$ the full subcategory of $(\mr{fs}/X)$ consisting of strict fs log schemes over $X$. Note that the category $(\mr{st}/X)$ is canonically identified with the category of the schemes over the underlying scheme of $X$.

Kato's multiplicative group (or the log multiplicative group) $\Gml$ is the sheaf on $\Xet$ defined by $\Gml(U)=\Gamma(U,M^{\mr{gp}}_U)$ for any $U\in\fsX$, where $M^{\mr{gp}}_U$ denotes the group envelope of the log structure $M_U$ of $U$. The classical \'etale sheaf $\Gml$ is also a sheaf on $\Xket$ and $X_{\mr{kfl}}$ respectively, see \cite[Cor. 2.22]{niz1} for a proof.

By convention, for any sheaf of abelian groups $F$ on $\Xkfl$ and a subgroup sheaf $G$ of $F$ on $\Xkfl$, we denote by $(F/G)_{\Xfl}$ the quotient sheaf on $\Xfl$, while $F/G$ denotes the quotient sheaf on $\Xkfl$. We abbreviate the quotient sheaf $\Gml/\Gm$ on $\Xkfl$ as $\Gmlb$.

Let $F$ be a sheaf on $X_{\mr{\acute{e}t}}$ (resp. $X_{\mr{fl}}$, resp. $X_{\mr{k\acute{e}t}}$, resp. $X_{\mr{kfl}}$), and $U\in (\mr{fs}/X)$, we denote by $H_{\mr{\acute{e}t}}^i(U,F)$ (resp. $H_{\mr{fl}}^i(U,F)$, resp. $H_{\mr{k\acute{e}t}}^i(U,F)$, resp. $H_{\mr{kfl}}^i(U,F)$) the $i$-th sheaf cohomology group of the sheaf $F$ on $X_{\mr{\acute{e}t}}$ (resp. $X_{\mr{fl}}$, resp. $X_{\mr{k\acute{e}t}}$, resp. $X_{\mr{kfl}}$). For a sheaf $F$ on $X_{\mr{fl}}=(\mr{fs}/X)_{\mr{fl}}$, the canonical map $(\mr{fs}/X)_{\mr{fl}}\rightarrow (\mr{st}/X)_{\mr{fl}}$ of sites satisfies the conditions from \cite[\href{https://stacks.math.columbia.edu/tag/00XU}{Tag 00XU}]{stacks-project}, hence by \cite[\href{https://stacks.math.columbia.edu/tag/03YU}{Tag 03YU}]{stacks-project} the cohomology of the sheaf $F$ on $(\mr{fs}/X)_{\mr{fl}}$ can be computed on the smaller site $(\mr{st}/X)_{\mr{fl}}$. Hence for $U\in (\mr{st}/X)$, we will use the same notation $H_{\mr{fl}}^i(U,F)$ for the cohomology of $F$ on any of the two sites $(\mr{fs}/X)_{\mr{fl}}$ and $(\mr{st}/X)_{\mr{fl}}$.

Let $G$ be a group scheme over $X$, and $F$ a presheaf of abelian groups on $(\mr{fs}/X)$ endowed with a $G$-action. We denote by $H^i_X(G,F)$ the $i$-th cohomology group of the $G$-module $F$, and see \cite[Chap. II, \S 3]{d-g1} for the definition of such cohomology groups. For an abstract group $\Gamma$ and a $\Gamma$-module $N$, we denote by $H^i(\Gamma,N)$ the $i$-th cohomology group of the $\Gamma$-module $N$.

\section*{Introduction}
According to \cite[\S 0]{ill2}, the notion of log structure is due to L. Illusie and J. M. Fontaine, and the theory of log structure is essentially carried out by K. Kato. One merit of log geometry is that certain non-smooth morphisms in classical algebraic geometry become smooth (more precisely log smooth) in the world of log geometry. There is also the notion of log \'etale morphism, as well as the notion of log flat morphism in log geometry. Using log \'etale morphisms (resp. Kummer log \'etale morphisms), K. Fujiwara, K. Kato, and C. Nakayama developed a theory of log \'etale topology (resp. Kummer log \'etale topology), see \cite{ill1} for an overview. Using Kummer log flat morphisms, K. Kato developed a theory of Kummer log flat topology, see \cite{kat2} as well as \cite{niz1}. We will call Kummer log \'etale (resp. Kummer log flat) simply by Kummer \'etale (resp. Kummer flat).
 
In order to understand the cohomology theory in the Kummer \'etale topology (resp. Kummer flat topology), one is led to the comparison of the Kummer \'etale topology (resp. Kummer flat topology) with the classical \'etale topology (resp. classical flat topology). The main topic of this paper is to investigate such comparisons, i.e. to investigate the functor $R^i\varepsilon_{\mr{\acute{e}t}*}$ (resp. $R^i\varepsilon_{\mr{fl}*}$) associated to the map of sites (\ref{eq0.1}) (resp. (\ref{eq0.2})) for $i>0$.

\subsection*{Main results and Outline}
In the first half of Subsection \ref{subsec1.1}, we study the higher direct images $R^i\varepsilon_{\mr{\acute{e}t}*}G$ for $G$ a smooth commutative group scheme with connected fibres over the base. The main result is the following Theorem, see also Theorem \ref{thm1.1}.

\begin{thm}\label{thm0.1}
Let $X$ be a locally  noetherian fs log scheme, let $G$ be a smooth commutative group scheme with connected fibers over the underlying scheme of $X$, and let $i$ be a positive integer. 
\begin{enumerate}[(1)]
\item We have $$R^i\varepsilon_{\mr{\acute{e}t}*}G=\varinjlim_n (R^i\varepsilon_{\mr{\acute{e}t}*}G)[n]=\bigoplus_{\text{$l$ prime}}(R^i\varepsilon_{\mr{\acute{e}t}*}G)[l^\infty],$$
where $(R^i\varepsilon_{\mr{\acute{e}t}*}G)[n]$ denotes the $n$-torsion subsheaf of $R^i\varepsilon_{\mr{\acute{e}t}*}G$ and $(R^i\varepsilon_{\mr{\acute{e}t}*}G)[l^\infty]$ denotes the $l$-primary part of $R^i\varepsilon_{\mr{\acute{e}t}*}G$ for a prime number $l$.
\item The $l$-primary part $(R^i\varepsilon_{\mr{\acute{e}t}*}G)[l^\infty]$ is supported on the locus where $l$ is invertible.
\item If $n$ is invertible on $X$, then 
$$(R^i\varepsilon_{\mr{\acute{e}t}*}G)[n]=R^i\varepsilon_{\mr{\acute{e}t}*}G[n]=G[n](-i)\otimes_{\Z}\bigwedge^i(\Gml/\Gm)_{X_{\mr{\acute{e}t}}}.$$
\end{enumerate}
\end{thm}

When the underlying scheme of the base $X$ is over $\Q$, Theorem \ref{thm0.1} actually gives an explicit description of $R^i\varepsilon_{\mr{\acute{e}t}*}G$, see Corollary \ref{cor1.6}. When the underlying scheme of the base $X$ is over a finite field, we also have an explicit description of $R^i\varepsilon_{\mr{\acute{e}t}*}G$, see Corollary \ref{cor1.7}.

In the second half of Subsection \ref{subsec1.1}, we study the first higher direct image $R^1\varepsilon_{\mr{fl}*}G$ for $G$ a smooth commutative group scheme over the base. The main result is the following Theorem, see also Theorem \ref{thm1.5}.

\begin{thm}\label{thm0.2}
Let $X$ be a locally noetherian fs log scheme, and let $G$ be either a finite flat group scheme over the underlying scheme of $X$ or a smooth commutative group scheme over the underlying scheme of $X$. We endow $G$ with the induced log structure from $X$. Then we have 
$$R^1\varepsilon_{\mr{fl}*} G=
\varinjlim_{n}\mc{H}om_{X}(\Z/n\Z(1),G)\otimes_{\Z}(\Gml/\Gm)_{X_{\mr{fl}}}.$$
\end{thm}

In Subsection \ref{subsec1.2}, we mainly study the 2nd higher direct image $R^2\varepsilon_{\mr{fl}*}G$ for a torus $G$ over the base. The main result is the following theorem, see also Theorem \ref{thm1.8}.

\begin{thm}\label{thm0.3}
Let $X$ be a locally  noetherian fs log scheme, and let $G$ be a torus over the underlying scheme of $X$. 
\begin{enumerate}[(1)]
\item We have 
$$R^2\varepsilon_{\mr{fl}*}G=\varinjlim_{n} (R^2\varepsilon_{\mr{fl}*}G)[n]=\bigoplus_{\text{$l$ prime}}(R^2\varepsilon_{\mr{fl}*}G)[l^\infty],$$
where $(R^2\varepsilon_{\mr{fl}*}G)[n]$ denotes the $n$-torsion subsheaf of $R^2\varepsilon_{\mr{fl}*}G$ and $(R^2\varepsilon_{\mr{fl}*}G)[l^\infty]$ denotes the $l$-primary part of $R^2\varepsilon_{\mr{fl}*}G$ for a prime number $l$.
\item We have $(R^2\varepsilon_{\mr{fl}*}G)[n]=R^2\varepsilon_{\mr{fl}*}G[n]$.
\item The $l$-primary part $(R^2\varepsilon_{\mr{fl}*}G)[l^\infty]$ is supported on the locus where $l$ is invertible.
\item If $n$ is invertible on $X$, then 
$$(R^2\varepsilon_{\mr{fl}*}G)[n]=R^2\varepsilon_{\mr{fl}*}G[n]=G[n](-2)\otimes_{\Z}\bigwedge^2(\Gml/\Gm)_{X_{\mr{fl}}}.$$
\end{enumerate}
\end{thm}

When the underlying scheme of $X$ is a $\Q$-scheme or an $\F_q$-scheme for a finite field $\F_q$, Theorem \ref{thm0.3} can be strengthened to give an explicit description of $R^2\varepsilon_{\mr{fl}*}G$, see Corollary \ref{cor1.11} and Corollary \ref{cor1.12} respectively. 

Another case in which we have an explicit description of $R^2\varepsilon_{\mr{fl}*}G$ is related to log structure. Let $Y\in (\mr{fs}/X)$ be such that the ranks of the stalks of the \'etale sheaf $M_Y/\mc{O}_Y^{\times}$ are at most one, and let $(\mr{st}/Y)$ be the full subcategory of $(\mr{fs}/X)$ consisting of strict fs log schemes over $Y$. Then the restriction of $R^2\varepsilon_{\mr{fl}*}G$ to $(\mr{st}/Y)$ is zero, here $G$ is as in Theorem \ref{thm0.3}.

Theorem \ref{thm0.3} can also be slightly generalized to include a certain class of unipotent group schemes, see Theorem \ref{thm1.9}.

In Section \ref{sec2}, we investigate the higher direct image of Kato's logarithmic multiplicative group $\Gml$. The main result is the following theorem, see also Theorem \ref{thm2.1}.  

\begin{thm}\label{thm0.4}
Let $X$ be an fs log scheme with its underlying scheme locally  noetherian. Then we have:
\begin{enumerate}[(1)]
\item $R^r\varepsilon_{\mr{fl}*}\Gmlb=0$ (resp. $R^r\varepsilon_{\mr{\acute{e}t}*}\Gmlb=0$) for $r\geq 1$;
\item the canonical map $R^r\varepsilon_{\mr{fl}*}\Gm\rightarrow R^r\varepsilon_{\mr{fl}*}\Gml$ (resp. $R^r\varepsilon_{\mr{\acute{e}t}*}\Gm\rightarrow R^r\varepsilon_{\mr{\acute{e}t}*}\Gml$) is an isomorphism for $r\geq 2$.
\end{enumerate}
\end{thm}

By Theorem \ref{thm0.2}, we have $R^1\varepsilon_{\mr{fl}*}\Gm=
\Q/\Z\otimes_{\Z}(\Gml/\Gm)_{X_{\mr{fl}}}$. Theorem \ref{thm0.1} gives a description of $R^1\varepsilon_{\mr{\acute{e}t}*}\Gm$. In contrast to Theorem \ref{thm0.4} (2), the canonical maps $R^1\varepsilon_{\mr{fl}*}\Gm\rightarrow R^1\varepsilon_{\mr{fl}*}\Gml$ and $R^1\varepsilon_{\mr{\acute{e}t}*}\Gm\rightarrow R^1\varepsilon_{\mr{\acute{e}t}*}\Gml$ are not isomorphisms in general by the following theorem of Kato (see \cite[Cor. 5.2]{kat2}).

\begin{thm}[Kato]\label{thm0.5}
Let $X$ be an fs log scheme with its underlying scheme locally noetherian. Then we have $R^1\varepsilon_{\mr{fl}*}\Gml=0$ and $R^1\varepsilon_{\mr{\acute{e}t}*}\Gml=0$.
\end{thm}

In Section \ref{sec3}, we apply the previous results to compute $H^i_{\mr{kfl}}(X,\Gm)$ for $i=1,2$ in the following two cases.
\begin{enumerate}[(1)]
\item $R$ is a discrete valuation ring with finite residue field, and $X=\Spec R$ endowed with the canonical log structure associated to its closed point.
\item $K$ is either a number field or a function field. When $K$ is a number field, $X$ is the spectrum of the ring of integers of $K$. When $K$ is a function field, $X$ is the unique smooth projective curve associated to $K$. Let $S$ be a finite set of closed points of $X$, $U:=X-S$, and $j:U\hookrightarrow X$. We endow $X$ with the log structure $j_*\mc{O}_U^{\times}\cap\mc{O}_X\rightarrow\mc{O}_X$. 
\end{enumerate}

\subsection*{History}
As we have stated in the beginning, Kummer flat topology is introduced by K. Kato in \cite{kat2}. The comparison between the Kummer flat topology and the classical flat topology is initiated in loc. cit.. In fact our Theorem \ref{thm0.2} is a generalization of \cite[Thm. 4.1]{kat2} from smooth affine commutative group schemes to smooth commutative group schemes. Theorem \ref{thm0.5} is also from \cite[Cor. 5.2]{kat2}. We would like to point out that \cite{kat2} was started around 1991 and has been circulated as an incomplete preprint for a long time until 2019. The comparison results in loc. cit. has been reproduced by Niziol in \cite{niz1} with proofs. Besides \cite{kat2}, the only existing comparison theorem between the Kummer log topologies and the classical topologies is the following one from \cite{k-n1}.

\begin{thm}[Kato-Nakayama]\label{thm0.6}
Let $X$ be an fs log scheme, and let $F$ be a sheaf of abelian groups on $(\mr{fs}/X)_{\mr{\acute{e}t}}$ such that 
$$F=\bigcup_{n\colon\text{invertible on $X$}}\mr{Ker}(F\xrightarrow{n}F).$$
Then the cup-product induces an isomorphism
$$F(-q)\otimes_{\Z}\bigwedge^q(\Gml/\Gm)_{X_{\mr{\acute{e}t}}}\rightarrow R^q\varepsilon_{\mr{\acute{e}t}*}\varepsilon_{\mr{\acute{e}t}}^*F$$
for any $q\geq 0$, where $(-q)$ denotes the Tate twist.
\end{thm}

\subsection*{Ideas of the proofs}
Like what happens very often in cohomological comparison, \v{C}ech cohomology is a key tool for the proofs of the comparison results in this article.

For Theorem \ref{thm0.1} (resp. Theorem \ref{thm0.3}), we first prove that $R^i\varepsilon_{\mr{\acute{e}t}*}G$ (resp. $R^2\varepsilon_{\mr{fl}*}G$) is torsion and $R^{i-1}\varepsilon_{\mr{\acute{e}t}*}G$ (resp. $R^1\varepsilon_{\mr{fl}*}G$) is divisible. If $n$ is invertible on the base (resp. $n$ is arbitrary), we have a short exact sequence $0\to G[n]\to G\xrightarrow{n}G\to0$ on $(\mr{fs}/X)_{\mr{k\acute{e}t}}$ (resp. $(\mr{fs}/X)_{\mr{kfl}}$). And this short exact sequence further gives rise to a short exact sequence 
\begin{align*}
0\to R^{i-1}\varepsilon_{\mr{\acute{e}t}*}G\otimes_{\Z}\Z/n\Z \to R^i\varepsilon_{\mr{\acute{e}t}*}G[n]\to (R^i\varepsilon_{\mr{\acute{e}t}*}G)[n]\to0  \\
\text{(resp.  $0\to R^{1}\varepsilon_{\mr{fl*}}G\otimes_{\Z}\Z/n\Z \to R^2\varepsilon_{\mr{fl*}}G[n]\to (R^2\varepsilon_{\mr{fl}*}G)[n]\to0$).}
\end{align*}
Then we are reduced to investigate $R^i\varepsilon_{\mr{\acute{e}t}*}G[n]$ (resp. $R^2\varepsilon_{\mr{fl}*}G[n]$). Clearly the sheaf $R^i\varepsilon_{\mr{\acute{e}t}*}G[n]$ can be explicitly described via Theorem \ref{thm0.6}. And actually the sheaf $R^1\varepsilon_{\mr{fl}*}G[n]$ can also be made well-understood via Theorem \ref{thm0.6} with some  effort, see Corollary \ref{cor1.9}. Part (2) of Theorem \ref{thm0.1} and part (3) of Theorem \ref{thm0.3} are reduced to local computations, see Corollary \ref{cor1.4} and Lemma \ref{lem1.6} (3).

The proof of Theorem \ref{thm0.2} is a modification of that of \cite[Thm. 4.1]{kat2} which is the affine case of Theorem \ref{thm0.2}. The author loc. cit. first constructs a canonical homomorphism
\begin{equation}\label{eq0.3}
\varinjlim_{n}\mc{H}om_{X}(\Z/n\Z(1),G)\otimes_{\Z}(\Gml/\Gm)_{X_{\mr{fl}}}\to R^1\varepsilon_{\mr{fl}*} G.
\end{equation}
Then it suffices to show that this homomorphism is actually an isomorphism. One is reduced to check this for $X=\Spec A$ such that $A$ is a noetherian strictly Henselian local ring and $X$ admits a chart $P\to M_X$ with $P$ an fs monoid such that the canonical map $P\xrightarrow{\cong}M_{X,x}/\mc{O}_{X,x}^\times$ is an isomorphism, where $x$ denotes the closed point of $X$. Therefore one is led to the computation of $H_{\mr{kfl}}^1(X,G)$ in this case. For any positive integer $m$, we define
\begin{equation}\label{eq0.4}
X_m:=X\times_{\Spec\Z[P]}\Spec\Z[P^{1/m}]
\end{equation}
endowed with the canonical log structure coming from $P^{1/m}$, where $P^{1/m}$ is a monoid endowed with a homomorphism $P\rightarrow P^{1/m}$ which can be identified with $P\xrightarrow{m}P$. We get a Kummer flat cover $X_m\rightarrow X$, which is also a Kummer \'etale cover if $m$ is coprime to the characteristic of the residue field of $A$. It is shown in \cite[Lem. 4.6]{kat2} and its proof that 
$$H_{\mr{kfl}}^1(X,G)=\check{H}^1_{\mr{kfl}}(X,G)=\varinjlim_m\check{H}^1_{\mr{kfl}}(X_m/X,G)=\varinjlim_mH^1(C_{G,m}),$$
where $C_{G,m}$ is the \v{C}ech complex for the cover $X_m\to X$ with coefficients in $G$. In order to show that (\ref{eq0.3}) is an isomorphism, one is reduced to show that
\begin{equation}\label{eq0.5}
\mr{Hom}_{X}(\Z/n\Z(1),G)\otimes_{\Z}P^{\mr{gp}}\to H^1(C_{G,n})
\end{equation}
is an isomorphism for each positive integer $n$, see \cite[\S 4.9]{kat2} for the details of the reduction. For this, the author first deal with the case that $A$ is complete in \cite[Prop. 4.10]{kat2}. And the proof of \cite[Prop. 4.10]{kat2} is first given in the artinian case in \cite[\S 4.11]{kat2}, then achieved in general by passing to limit in \cite[\S 4.13]{kat2}. The proof of (\ref{eq0.5}) in the general case is reduced to the complete case by descent which requires the condition that $G$ is affine. In this article, we make the following key lemma (see also Lemma \ref{lem1.3}), which reduces the computation of $\check{H}^1_{\mr{kfl}}(X_m/X,G)$ directly to the artinian case. Therefore we can remove the condition that $G$ is affine.
\begin{lem}[Key Lemma]
Let $X$ be an fs log scheme whose underlying scheme is $\Spec A$ with $A$ a noetherian strictly Henselian local ring, and let $x$ be the closed point of $X$. Let $P\to M_X$ be a chart of $X$ with $P$ an fs monoid such that the induced map $P\xrightarrow{\cong}M_{X,x}/\mc{O}_{X,x}^\times$ is an isomorphism. Let $X_m$ be as constructed in (\ref{eq0.4}). We regard $x$ as an fs log scheme with respect to the induced log structure, and $x_m:=X_m\times_Xx$. Let ? be either kfl or k\'et, then the canonical map 
$$\check{H}_{?}^i(X_m/X,G)\rightarrow \check{H}_{?}^i(x_m/x,G)$$
is an isomorphism for all $i>0$.
\end{lem}

At last, we briefly discuss the proof of Theorem \ref{thm0.4}. The second part of Theorem \ref{thm0.4} clearly follows from the first part. And part (1), i.e. the vanishings of $R^r\varepsilon_{\mr{fl}*}\Gmlb$ and $R^r\varepsilon_{\mr{\acute{e}t}*}\Gmlb$ for $r\geq 1$, follow from the vanishings of $H_{\mr{kfl}}^r(X,\Gmlb)$ and $H_{\mr{k\acute{e}t}}^r(X,\Gmlb)$ with $X$ having its underlying scheme the spectrum of a noetherian strictly Henselian local ring. The latter are given by explicit computations of \v{C}ech cohomology, see Theorem \ref{thm2.1} and Lemma \ref{lem2.1}.

\section{The higher direct images for smooth commutative group schemes}\label{sec1}
In this section, we investigate the higher direct images under $\varepsilon_{\mr{fl}}$ and $\varepsilon_{\mr{\acute{e}t}}$ for smooth commutative group schemes. We deal with the case of $\varepsilon_{\mr{\acute{e}t}}$ in the first subsection, and with the case of $\varepsilon_{\mr{fl}}$ in the second subsection. In the Kummer \'etale case, we have more tools at hand and can get results for $R^i\varepsilon_{\mr{\acute{e}t}*}G$ for $i\geq2$ and $G$ a smooth commutative group scheme with connected fibres over the base. In the Kummer flat case, we  have only results for $R^2\varepsilon_{\mr{fl}*}$ for certain smooth commutative group schemes.
\subsection{Kummer \'etale case}\label{subsec1.1}
In this subsection, we study the higher direct image $R^i\varepsilon_{\mr{\acute{e}t}*} G$ along the forgetful map $\varepsilon_{\mr{\acute{e}t}}:(\mr{fs}/X)_{\mr{k\acute{e}t}}\rightarrow (\mr{fs}/X)_{\mr{\acute{e}t}}$, where $i\geq2$ and $G$ is a smooth commutative group scheme with connected fibres over $X$ regarded as a sheaf on $(\mr{fs}/X)_{\mr{k\acute{e}t}}$. 

In order to understand $R^i\varepsilon_{\mr{\acute{e}t}*}G$, we need to compute $H^i_{\mr{k\acute{e}t}}(X,G)$ first, for the case that $X$ has its underlying scheme $\Spec A$ with $A$ a noetherian strictly Henselian ring. We will make heavy use of \v{C}ech cohomology, for which we often refer to \cite[Chap. III]{mil1}. 

The following proposition is an analogue of \cite[Chap. III, Prop. 2.9]{mil1}.
\begin{prop}\label{prop1.1}
Let ? be either fl or \'et. Let $X$ be an fs log scheme, and $F$ a sheaf on $(\mr{fs}/X)_{\mr{k}?}$. Let $\underline{H}_{\mr{k}?}^i(F)$ be the presheaf $U\mapsto H_{\mr{k}?}^i(U,F)$ for $U\in (\mr{fs}/X)$. Then we have that the 0-th \v{C}ech cohomology group $\check{H}^0_{\mr{k}?}(U,\underline{H}_{\mr{k}?}^i(F))$ vanishes for $i>0$ and all $U\in(\mr{fs}/X)$. 
\end{prop}
\begin{proof}
The proof of \cite[Chap. III, Prop. 2.9]{mil1} is purely formal, and it works also here.
\end{proof}

The following corollary is an analogue of \cite[Chap. III, Cor. 2.10]{mil1}.

\begin{cor}\label{cor1.1}
Let ?, $X$, $F$ and $U$ be as in Proposition \ref{prop1.1}. Then the \v{C}ech cohomology to derived functor cohomology spectral sequence 
$$\check{H}_{\mr{k}?}^i(U,\underline{H}_{\mr{k}?}^j(F))\Rightarrow H^{i+j}_{\mr{k}?}(U,F)$$
induces isomorphisms 
$$\check{H}_{\mr{k}?}^i(U,F)\xrightarrow{\cong}H_{\mr{k}?}^i(U,F)$$
for $i=0,1$, and an exact sequence
$$0\rightarrow \check{H}_{\mr{k}?}^2(U,F)\rightarrow H_{\mr{k}?}^2(U,F)\rightarrow \check{H}_{\mr{k}?}^1(U,\underline{H}_{\mr{k}?}^1(F))\rightarrow \check{H}_{\mr{k}?}^3(U,F)\rightarrow H_{\mr{k}?}^3(U,F).$$
\end{cor}
\begin{proof}
The results follow from Proposition \ref{prop1.1}.
\end{proof}

\begin{cor}\label{cor1.2}
Let $X=\Spec A$ be an fs log scheme with $A$ a noetherian strictly Henselian local ring, $x$ the closed point of $X$, $p$ the characteristic of the residue field of $A$, and $F$ a sheaf on $(\mr{fs}/X)_{\mr{k\acute{e}t}}$. Let $P\xrightarrow{\alpha} M_X$ be a chart of the log structure of $X$ with $P$ an fs monoid, such that the induced map $P\rightarrow M_{X,x}/\mc{O}_{X,x}^{\times}$ is an isomorphism. For any positive integer $m$, we define $X_m$ to be the fs log scheme $X\times_{\Spec\Z[P]}\Spec\Z[P^{1/m}]$ endowed with the canonical log structure coming from $P^{1/m}$, where $P^{1/m}$ is a monoid endowed with a homomorphism $P\rightarrow P^{1/m}$ which can be identified with $P\xrightarrow{m}P$. We get a Kummer flat cover $f_m:X_m\rightarrow X$, which is also a Kummer \'etale cover if $m$ is coprime to $p$.

Let $\gamma\in  H_{\mr{k\acute{e}t}}^i(X,F)$, then there exists a positive integer $n$ with $(n,p)=1$ such that $\gamma$ maps to zero in $H_{\mr{k\acute{e}t}}^i(X_n,F)$ along $f_n:X_n\rightarrow X$.
\end{cor}
\begin{proof}
By Proposition \ref{prop1.1}, there exists a Kummer \'etale cover $$\{Y_i\xrightarrow{g_i} X\}_{i\in I}$$
such that $g_i^*(\gamma)=0$ for each $i\in I$. Let $i_0$ be such that $g_{i_0}(Y_{i_0})$ contains the closed point of $X$. By \cite[Prop. 2.15]{niz1}, there exists a commutative diagram
$$\xymatrix{
Z\ar[r]^h\ar[d]_g & X_n\ar[d]^{f_n}  \\
Y_{i_0}\ar[r]^{g_{i_0}} &X
},$$
such that the image of $g_{i_0}\circ g$ contains the closed point of $X$, $g$ is Kummer \'etale, $n$ is a positive integer which is invertible on $X$, and $h$ is classically \'etale. Then we have $h^*f_n^*\gamma=g^*g_{i_0}^*\gamma=0$. Since $h$ is classically \'etale and the underlying scheme of $X_n$ is strictly Henselian local, $h$ has a section $s$. It follows that 
$$f_n^*\gamma=s^*h^*(f_n^*\gamma)=s^*0=0.$$
\end{proof}

\begin{cor}\label{cor1.3}
Let the notation and the assumptions be as in Corollary \ref{cor1.2}. Then we have 
$$H_{\mr{k\acute{e}t}}^i(X,F)\cong\varinjlim_n\mr{ker}(H_{\mr{k\acute{e}t}}^i(X,F)\xrightarrow{f_n^*}H_{\mr{k\acute{e}t}}^i(X_n,F)).$$
\end{cor}
\begin{proof}
This follows from Corollary \ref{cor1.2}.
\end{proof}

\begin{prop}\label{prop1.2}
Let the notation and the assumptions be as in Corollary \ref{cor1.2}. We further let $\N':=\{n\in\N\mid (n,p)=1 \}$. 
\begin{enumerate}[(1)]
\item The family $\mathscr{X}_{\N}:=\{X_n\rightarrow X\}_{n\geq1}$ (resp. $\mathscr{X}_{\N'}:=\{X_n\rightarrow X\}_{n\in\N'}$) of Kummer flat covers (resp. Kummer \'etale covers) of $X$ satisfies the condition (L3) from \cite[\S 2]{art3}, whence a spectral sequence
\begin{equation}\label{eq1.1}
\check{H}_{\mr{kfl}}^i(\mathscr{X}_{\N},\underline{H}_{\mr{kfl}}^j(F))\Rightarrow H^{i+j}_{\mr{kfl}}(X,F)\quad (\text{resp. $\check{H}_{\mr{k\acute{e}t}}^i(\mathscr{X}_{\N'},\underline{H}_{\mr{k\acute{e}t}}^j(F))\Rightarrow H^{i+j}_{\mr{k\acute{e}t}}(X,F)$}),
\end{equation}
where 
$$\check{H}_{\mr{kfl}}^i(\mathscr{X}_{\N},F):=\varinjlim_{n\in\N}\check{H}_{\mr{kfl}}^i(X_n/X,F)$$
and
$$\check{H}_{\mr{k\acute{e}t}}^i(\mathscr{X}_{\N'},F):=\varinjlim_{n\in\N'}\check{H}_{\mr{k\acute{e}t}}^i(X_n/X,F).$$
\item We have $\check{H}_{\mr{k\acute{e}t}}^0(\mathscr{X}_{\N'},\underline{H}_{\mr{k\acute{e}t}}^j(F))=0$ for any $j>0$.
\end{enumerate}
\end{prop}
\begin{proof}
Part (1) follows from \cite[Chap. II, Sec. 3, (3.3)]{art3}.

Part (2) follows from Corollary \ref{cor1.2}. Indeed we have
$$\check{H}_{\mr{k\acute{e}t}}^0(\mathscr{X}_{\N'},\underline{H}_{\mr{k\acute{e}t}}^j(F))\hookrightarrow\varinjlim_{n\in\N'}H_{\mr{k\acute{e}t}}^j(X_n,F),$$
and the latter vanishes by Corollary \ref{cor1.2}.
\end{proof}

\begin{cor}\label{cor1.4}
Let the notation and the assumptions be as in Corollary \ref{cor1.2}. Then the groups $H_{\mr{k\acute{e}t}}^i(X,F)$ are torsion and $p$-torsion-free for all $i>0$.
\end{cor}
\begin{proof}
By the Kummer \'etale spectral sequence from (\ref{eq1.1}) and Proposition \ref{prop1.2} (2), it suffices to show that the groups $\varinjlim_{n\in\N'}\check{H}_{\mr{k\acute{e}t}}^i(X_n/X,\underline{H}_{\mr{k\acute{e}t}}^j(F))$ are torsion and $p$-torsion-free for all $i>0$ and $j\geq0$.  Let $H_n:=\Spec\Z[(P^{1/n})^{\mr{gp}}/P^{\mr{gp}}]$ which is a group scheme over $\Spec\Z$ such that 
$$X_n\times_XX_n=X_n\times_{\Spec\Z}H_n,$$
see \cite[the second paragraph on p522]{niz1} for more detailed descriptions of $H_n$. Regarded as a group scheme over $X$, $H_n$ is constant. And $X_n$ is a Galois cover of $X$ with Galois group $H_n$. By \cite[Example 2.6]{mil1}, we have 
$$\check{H}_{\mr{k\acute{e}t}}^i(X_n/X,\underline{H}_{\mr{k\acute{e}t}}^j(F))=H^i(H_n,H_{\mr{k\acute{e}t}}^j(X_n,F))$$
which is torsion and $p$-torsion-free for $i>0$. This finishes the proof.
\end{proof}

\begin{cor}\label{cor1.5}
Let $X$ be a locally  noetherian fs log scheme, and let $F$ be a sheaf on $(\mr{fs}/X)_{\mr{k\acute{e}t}}$. Then the sheaves $R^i\varepsilon_{\mr{\acute{e}t}*}F$ are torsion for $i>0$.
\end{cor}
\begin{proof}
This follows from Corollary \ref{cor1.4}.
\end{proof}

\begin{thm}\label{thm1.1}
Let $X$ be a locally  noetherian fs log scheme, let $G$ be a smooth commutative group scheme with connected fibers over the underlying scheme of $X$, and let $i$ be a positive integer. 
\begin{enumerate}[(1)]
\item We have $$R^i\varepsilon_{\mr{\acute{e}t}*}G=\varinjlim_n (R^i\varepsilon_{\mr{\acute{e}t}*}G)[n]=\bigoplus_{\text{$l$ prime}}(R^i\varepsilon_{\mr{\acute{e}t}*}G)[l^\infty],$$
where $(R^i\varepsilon_{\mr{\acute{e}t}*}G)[n]$ denotes the $n$-torsion subsheaf of $R^i\varepsilon_{\mr{\acute{e}t}*}G$ and $(R^i\varepsilon_{\mr{\acute{e}t}*}G)[l^\infty]$ denotes the $l$-primary part of $R^i\varepsilon_{\mr{\acute{e}t}*}G$ for a prime number $l$.
\item The $l$-primary part $(R^i\varepsilon_{\mr{\acute{e}t}*}G)[l^\infty]$ is supported on the locus where $l$ is invertible.
\item If $n$ is invertible on $X$, then 
$$(R^i\varepsilon_{\mr{\acute{e}t}*}G)[n]=R^i\varepsilon_{\mr{\acute{e}t}*}G[n]=G[n](-i)\otimes_{\Z}\bigwedge^i(\Gml/\Gm)_{X_{\mr{\acute{e}t}}}.$$
\end{enumerate}
\end{thm}
\begin{proof}
Part (1) follows from Corollary \ref{cor1.5}.

Part (2) follows from Corollary \ref{cor1.4}.

We are left with part (3). Since $n$ is invertible on $X$, the sequence 
\begin{equation}\label{eq1.2}
0\rightarrow G[n]\rightarrow G\xrightarrow{n}G\rightarrow0
\end{equation}
is a short exact sequence of sheaves of abelian groups for the classical flat topology and $G[n]$ is quasi-finite and \'etale. Indeed, if $X$ is a point, then by the structure theorem of connected algebraic groups, see \cite[Thm. 2.3, Thm. 2.4]{bri1}, we are reduced to check the cases that $G$ is a torus, or connected unipotent group, or abelian variety, which are clearly true. In general, it suffices to show that $G\xrightarrow{n}G$ is an epimorphism for the classical flat topology. This is clear, since it is set-theoretically surjective and flat by the fiberwise criterion of flatness, therefore it is faithfully flat. Moreover, $G[n]$ being \'etale over $X$ implies that $G\xrightarrow{n}G$ is even an epimorphism for the classical \'etale topology, and thus the sequence (\ref{eq1.2}) is also exact for the classical \'etale topology. Since the pullback functor $\varepsilon_{\mr{\acute{e}t}}^*$ is exact, the sequence (\ref{eq1.2}) remains exact on $(\mr{fs}/X)_{\mr{k\acute{e}t}}$ and induces a long exact sequence 
$$\rightarrow R^{i-1}\varepsilon_{\mr{\acute{e}t}*}G\xrightarrow{n} R^{i-1}\varepsilon_{\mr{\acute{e}t}*}G\rightarrow R^i\varepsilon_{\mr{\acute{e}t}*}G[n]\rightarrow R^i\varepsilon_{\mr{\acute{e}t}*}G\xrightarrow{n} R^i\varepsilon_{\mr{\acute{e}t}*}G.$$
This further induces a short exact sequence
$$0\rightarrow R^{i-1}\varepsilon_{\mr{\acute{e}t}*}G\otimes_{\Z}\Z/n\Z\rightarrow R^{i}\varepsilon_{\mr{\acute{e}t}*}G[n]\rightarrow (R^{i}\varepsilon_{\mr{\acute{e}t}*}G)[n]\rightarrow 0$$
for each $i>0$. We have $R^i\varepsilon_{\mr{\acute{e}t}*}G[n]=G[n](-i)\otimes_{\Z}\bigwedge^i(\Gml/\Gm)_{X_{\mr{\acute{e}t}}}$ by Theorem \ref{thm0.6}. To finish the proof, it suffices to prove that the sheaf $R^{i-1}\varepsilon_{\mr{\acute{e}t}*}G$ is $n$-divisible. We proceed by induction. For $i=1$, this is clear, since $G\xrightarrow{n}G$ is an epimorphism of sheaves of abelian groups for the classical \'etale topology. Assume $R^{j}\varepsilon_{\mr{\acute{e}t}*}G$ is $n$-divisible, then we have $R^{j+1}\varepsilon_{\mr{\acute{e}t}*}G[n]\xrightarrow{\cong} (R^{j+1}\varepsilon_{\mr{\acute{e}t}*}G)[n]$. Therefore
\begin{align*}
\varinjlim_r(R^{j+1}\varepsilon_{\mr{\acute{e}t}*}G)[n^r]=&\varinjlim_rR^{j+1}\varepsilon_{\mr{\acute{e}t}*}G[n^r] \\
=&\varinjlim_rG[n^r](-j-1)\otimes_{\Z}\bigwedge^{j+1}(\Gml/\Gm)_{X_{\mr{\acute{e}t}}},
\end{align*}
where the second equality follows from Theorem \ref{thm0.6}. Hence $\varinjlim_r(R^{j+1}\varepsilon_{\mr{\acute{e}t}*}G)[n^r]$ is $n$-divisible. It follows by part (1) that $R^{j+1}\varepsilon_{\mr{\acute{e}t}*}G$ is $n$-divisible. This finishes the induction.
\end{proof}

\begin{cor}\label{cor1.6}
Let $X$ be a locally noetherian fs log scheme such that the underlying scheme of $X$ is a $\Q$-scheme, and $G$ a smooth commutative group scheme with connected fibres over the underlying scheme of $X$. Then we have 
$$R^i\varepsilon_{\mr{\acute{e}t},*}G=\varinjlim_{n}G[n](-i)\otimes_{\Z}\bigwedge^i(\Gml/\Gm)_{X_{\acute{e}t}}.$$
\end{cor}
\begin{proof}
This follows from Theorem \ref{thm1.1}.
\end{proof}

\begin{cor}\label{cor1.7}
Let $p$ be a prime number. Let $X$ be a locally noetherian fs log scheme such that the underlying scheme of $X$ is an $\F_p$-scheme, and $G$ a smooth commutative group scheme with connected fibres over the underlying scheme of $X$. Then we have 
$$R^i\varepsilon_{\mr{\acute{e}t}*}G=\varinjlim_{(n,p)=1}G[n](-i)\otimes_{\Z}\bigwedge^i(\Gml/\Gm)_{X_{\acute{e}t}}.$$
\end{cor}
\begin{proof}
This follows Theorem \ref{thm1.1}.
\end{proof}

Now we are going to generalize Kato's description of $R^1\varepsilon_{\mr{fl}*}G$, see \cite[Thm. 4.1]{kat2}, from smooth affine group schemes to smooth group schemes. The following lemma is analogous to \cite[Proof of Thm. III.3.9, Step 2 plus Rmk. 3.11 (b)]{mil1}.

\begin{lem}\label{lem1.2}
Let $X$ be a locally noetherian fs log scheme endowed with a chart $P\rightarrow M_X$ with $P$ an fs monoid satisfying $P^{\times}=1$, and $G$ a smooth commutative group scheme over $X$ endowed with the induced log structure from $X$. For a positive integer $m$, we define $P^{1/m}$ and $f_m:X_m\rightarrow X$ as in Corollary \ref{cor1.2}. Let $H_m$ be the group scheme $\Spec\Z[(P^{1/m})^{\mr{gp}}/P^{\mr{gp}}]$ over $\Spec\Z$, then we have that the $(r+1)$-fold product $X_m\times_X\cdots\times_XX_m$ is isomorphic to $X_m\times_{\Spec\Z} H_m^r$, where $H_m^r$ denotes the $r$-fold product of $H_m$ over $\Spec\Z$. 

We define $\underline{C}^{\cdot}(G)$ to be the complex of functors $\underline{C}^i(G):(\mr{st}/X)\rightarrow\mr{Ab}$ such that, for any $Y\in (\mr{st}/X)$, $\underline{C}^{\cdot}(G)(Y)$ is the \v{C}ech complex $C^{\cdot}(Y_m/Y,G)$ for the Kummer flat cover $Y_m:=Y\times_XX_m\rightarrow Y$. Write $\underline{Z}^i(G)$ for the functor 
$$(\mr{st}/X)\rightarrow\mr{Ab},Y\mapsto\mr{ker}(d^i:C^i(Y_m/Y,G)\rightarrow C^{i+1}(Y_m/Y,G)),$$
then $d^{i-1}:\underline{C}^{i-1}(G)\rightarrow\underline{Z}^i(G)$ is representable by a smooth morphism of algebraic spaces over $X$ for $i\geq1$.
\end{lem}
\begin{proof}
By definition $\underline{C}^{i}(G)$ is the functor
$$(\mr{st}/X)\rightarrow\mr{Ab}, Y\mapsto G(Y_m\times_Y\cdots\times_YY_m)=G(Y\times_X X_m\times_{\Spec\Z}H_m^i),$$
that is, it is $\mathring{\pi}_*G$, where $\pi$ denotes the map $X_m\times_{\Spec\Z}H_m^i\rightarrow X$ and $\mathring{\pi}$ denotes the underlying map of schemes of $\pi$. Since $\mathring{\pi}$ is clearly finite and faithful flat, $\underline{C}^{i}(G)$ is therefore represented by the Weil restriction of scalars of $G\times_X(X_m\times_{\Spec\Z}H_m^i)$, which is representable by a group algebraic space by \cite[Chap. V, 1.4 (a)]{mil1}. Hence the functor $\underline{Z}^i(G)$, as the kernel of a map $d^i:\underline{C}^i(G)\rightarrow \underline{C}^{i+1}(G)$ of group algebraic spaces over $X$, is representable by a group algebraic space over $X$.

Now we prove the smoothness of $d^{i-1}:\underline{C}^{i-1}(G)\rightarrow\underline{Z}^i(G)$. It suffices to show that, for any affine $X$-scheme $T$ and closed subscheme $T_0$ of $T$ defined by an ideal $I$ of square zero, any $z\in\underline{Z}^i(G)(T)$ whose image $z_0$ in $\underline{Z}^i(G)(T_0)$ arises from an element $c_0\in\underline{C}^{i-1}(G)(T_0)$, there exists $c\in\underline{C}^{i-1}(G)(T)$ such that $c$ maps to $c_0$. Let $N$ be the functor 
$$(\mr{Sch}/T)\rightarrow\mr{Ab},Y\mapsto\mr{ker}(G(Y)\rightarrow G(Y\times_TT_0)).$$
Let $C^{\cdot}(T_m/T,N)$ be the \v{C}ech complex for the Kummer flat cover $T_m:=T\times_XX_m\rightarrow T$ with coefficients in $N$.
Then we have the following commutative diagram
$$\xymatrix{
0\ar[r] &C^{i-1}(T_m/T,N)\ar[r]\ar[d] &\underline{C}^{i-1}(G)(T)\ar[r]\ar[d] &\underline{C}^{i-1}(G)(T_0)\ar[d]\ar[r] &0  \\
0\ar[r] &C^{i}(T_m/T,N)\ar[r]\ar[d] &\underline{C}^{i}(G)(T)\ar[r]\ar[d] &\underline{C}^{i}(G)(T_0)\ar[d]\ar[r] &0  \\
0\ar[r] &C^{i+1}(T_m/T,N)\ar[r] &\underline{C}^{i+1}(G)(T)\ar[r] &\underline{C}^{i+1}(G)(T_0)\ar[r] &0  \\
}$$
with exact rows, where the exactness property at the right hand side follows from the smoothness of $G$. Let $Z^i(T_m/T,N)$ be the kernel of 
$$C^{i}(T_m/T,N)\rightarrow C^{i+1}(T_m/T,N),$$
then the above diagram induces the following commutative diagram
$$\xymatrix{
0\ar[r] &C^{i-1}(T_m/T,N)\ar[r]\ar[d] &\underline{C}^{i-1}(G)(T)\ar[r]\ar[d] &\underline{C}^{i-1}(G)(T_0)\ar[d]\ar[r] &0  \\
0\ar[r] &Z^{i}(T_m/T,N)\ar[r] &\underline{Z}^{i}(G)(T)\ar[r] &\underline{Z}^{i}(G)(T_0) \\
}$$
with exact rows. By an easy diagram chasing, for the existence of $c$, it suffices to show that the map $C^{i-1}(T_m/T,N)\rightarrow Z^{i}(T_m/T,N)$ is surjective, that is, that $\check{H}^i(T_m/T,N)=0$, $i\geq 1$. 

To finish the proof, we compute $\check{H}^i(T_m/T,N)$ for $i\geq 1$. Let $\tilde{N}$ be the functor 
$$(\mr{Sch}/T)\rightarrow\mr{Ab},U\mapsto N(U\times_TT_m).$$
Since $T_m\times_TT_m\cong T_m\times H_m$, the group scheme $H_m$ acts on the functor $\tilde{N}$, and the \v{C}ech complex $C^{\cdot}(T_m/T,N)$ can be identified with the standard complex $C^{\cdot}(H_m,\tilde{N})$ computing the cohomology of the $H_m$-module $\tilde{N}$. We claim that $\tilde{N}$ is coherent, then the vanishing of $\check{H}^i(T_m/T,N)$ follows from $H_m$ being diagonalizable by \cite[Exp. I, Thm. 5.3.3]{sga3-1}. For any $U\in (\mr{Sch}/T)$, the smoothness of $G$ implies that 
\begin{align*}
\tilde{N}(U)=N(U\times_TT_m)=&\mr{ker}(G(U\times_TT_m)\rightarrow G(U\times_TT_m\times_TT_0))  \\
=&\mr{Lie}(G)\otimes_{\Gamma(T,\mc{O}_T)}\Gamma(T_m,I\mc{O}_{T_m})\otimes_{\Gamma(T,\mc{O}_T)}\Gamma(U,\mc{O}_U).
\end{align*}
Therefore $\tilde{N}$ is coherent.
\end{proof}

The following lemma is analogous to \cite[Proof of Thm. III.3.9, Step 3 plus Rmk. 3.11 (b)]{mil1}.

\begin{lem}[Key Lemma]\label{lem1.3}
Let $X, X_m, f_m, P, P^{1/m}$ be as in Lemma \ref{lem1.2}. We further assume that the underlying scheme of $X$ is $\Spec A$ with $A$ a Henselian local ring, and let $x$ be the closed point of $X$. We regard $x$ as an fs log scheme with respect to the induced log structure, and $x_m:=x\times_{\Spec\Z[P]}\Spec\Z[P^{1/m}]$ is obviously identified with $X_m\times_Xx$ canonically. Let ? be either kfl or k\'et, then the canonical map 
$$\check{H}_{?}^i(X_m/X,G)\rightarrow \check{H}_{?}^i(x_m/x,G)$$
is an isomorphism for all $i>0$.
\end{lem}
\begin{proof}
Since $G$ is smooth and $X_m\times_X\cdots\times_XX_m\cong X_m\times_{\Spec\Z} H_m^r$ is a disjoin union of spectra of Henselian local rings, the maps
$$C^{i}(X_m/X,G)\rightarrow C^{i}(x_m/x,G)$$
are surjective by \cite[Chap. I, 4.13]{mil1}. Thus we are reduced to show that the complex $\mr{ker}(C^{\cdot}(X_m/X,G)\rightarrow C^{\cdot}(x_m/x,G))$ is exact. Let 
$$z\in \mr{ker}(C^{i}(X_m/X,G)\xrightarrow{d^i} C^{i+1}(X_m/X,G))$$
have image $z_0=0$ in $C^{i}(x_m/x,G)$. We seek a $c\in C^{i-1}(X_m/X,G)$ with $c_0=0$ such that $d^{i-1}(c)=z$. By Lemma \ref{lem1.2}, we know that $(d^{i-1})^{-1}(z)$ is representable by a smooth scheme over $X$. But $(d^{i-1})^{-1}(z)$ has a section over $x$, namely the zero section, and as $X$ is Henselian, this lifts to a section of $(d^{i-1})^{-1}(z)$ over $X$. This finishes the proof.
\end{proof}

With the help of Lemma \ref{lem1.3}, we can give a slightly different proof of Kato's Theorem (\cite[Thm. 4.1]{kat2}) which describes $R^1\varepsilon_{\mr{fl}*}G$ for a smooth affine group scheme $G$. This alternative proof allows us to remove the affinity condition on $G$ from \cite[Thm. 4.1 (ii)]{kat2}.

\begin{thm}\label{thm1.4}
Let $X=\Spec A$ be an fs log scheme with $A$ a noetherian strictly Henselian local ring, $x$ the closed point of $X$, $p$ the characteristic of the residue field of $A$, and we fix a chart $P\rightarrow M_X$ satisfying $P\cong M_{X,x}/\mc{O}_{X,x}^{\times}$. Let ? be either fl or \'et, and $\varepsilon_?:(\mr{fs}/X)_{\mr{k}?}\rightarrow (\mr{fs}/X)_{?}$  the canonical ``forgetful'' map of sites, and let $G$ be a smooth commutative group scheme over $X$ endowed with the log structure induced from $X$. Then we have canonical isomorphisms 
$$H_{\mr{k}?}^1(X,G)\cong\check{H}_{\mr{k}?}^1(X,G)\cong
\begin{cases}
\varinjlim_{(n,p)=1}\mr{Hom}_{X}(\Z/n\Z(1),G)\otimes_{\Z}P^{\mr{gp}},&\text{ if ?=\'et}  \\
\varinjlim_{n}\mr{Hom}_{X}(\Z/n\Z(1),G)\otimes_{\Z}P^{\mr{gp}}, &\text{ if ?=fl}
\end{cases}.$$
\end{thm}
\begin{proof}
In the proof of \cite[Prop. 3.13]{niz1}, the affineness of $G$ is only used in the paragraph before \cite[Cor. 3.17]{niz1}. Its use is to extend from the complete local case (see \cite[Lem. 3.15]{niz1}) to the Henselian local case. Note that the proof of \cite[Lem. 3.15]{niz1} deals with artinian local case first, then pass to complete local case. With the help of Lemma \ref{lem1.3}, we can pass from artinian local case directly to Henselian local case, hence no affineness of $G$ is needed.
\end{proof}

\begin{thm}\label{thm1.5}
Let $X$ be a locally noetherian fs log scheme, and let $G$ be either a finite flat group scheme over the underlying scheme of $X$ or a smooth commutative group scheme over the underlying scheme of $X$. We endow $G$ with the induced log structure from $X$. Then we have 
$$R^1\varepsilon_{\mr{fl}*} G=
\varinjlim_{n}\mc{H}om_{X}(\Z/n\Z(1),G)\otimes_{\Z}(\Gml/\Gm)_{X_{\mr{fl}}}.$$
\end{thm}
\begin{proof}
The statement for the finite flat case is the same as in \cite[Thm. 4.1]{kat2}, we only need to deal with the other case, which follows from Theorem \ref{thm1.4} in the same way as \cite[Thm. 3.12]{niz1} follows from \cite[Prop. 3.13]{niz1}.
\end{proof}

The following proposition about $H_{\mr{k}?}^1(X,G)$ will be used in next subsection.

\begin{prop}\label{prop1.3}
Let the notation and the assumptions be as in Theorem \ref{thm1.4}. Then we have $\varinjlim_nH_{\mr{kfl}}^1(X_{n},G)=0$ and $\varinjlim_{(n,p)=1}H_{\mr{k\acute{e}t}}^1(X_{n},G)=0$
\end{prop}
\begin{proof}
We only prove $\varinjlim_nH_{\mr{kfl}}^1(X_{n},G)=0$, the other statement can be proven in the same way. By Theorem \ref{thm1.4}, we have 
$$H_{\mr{kfl}}^1(X_n,G)\cong \varinjlim\limits_r\mr{Hom}_{X_{n}}(\Z/r\Z(1),G)\otimes_{\Z}(P^{\frac{1}{n}})^{\mr{gp}}$$
for each $n>0$. As $n$ varies, these isomorphisms fit into commutative diagrams of the form
$$\xymatrix{
H_{\mr{kfl}}^1(X_n,G)\ar[r]\ar[d]_{\cong} &H_{\mr{kfl}}^1(X_{mn},G)\ar[d]^{\cong}  \\
\varinjlim\limits_{r}\mr{Hom}_{X_{n}}(\Z/r\Z(1),G)\otimes_{\Z}(P^{\frac{1}{n}})^{\mr{gp}} \ar[r] &\varinjlim\limits_{r}\mr{Hom}_{X_{mn}}(\Z/r\Z(1),G)\otimes_{\Z}(P^{\frac{1}{mn}})^{\mr{gp}} 
}$$
with the second row induced by the canonical inclusion $P^{1/n}\hookrightarrow P^{1/mn}$. The group $\mr{Hom}_{X_{n}}(\Z/r\Z(1),G)$ is clearly torsion, it follows that 
$$\varinjlim_{n}H_{\mr{kfl}}^1(X_n,G)=\varinjlim\limits_{n}\varinjlim_{r}\mr{Hom}_{X_{n}}(\Z/r\Z(1),G)\otimes_{\Z}(P^{\frac{1}{n}})^{\mr{gp}} =0.$$
\end{proof}

\subsection{Kummer flat case}\label{subsec1.2}
Throughout this subsection, $X$ is an fs log scheme with its underlying scheme locally noetherian, and $G$ is a smooth commutative group scheme over the underlying scheme of $X$. We are going to investigate the second higher direct image $R^2\varepsilon_{\mr{fl}*} G$ along the forgetful map $\varepsilon_{\mr{fl}}:(\mr{fs}/X)_{\mr{kfl}}\rightarrow (\mr{fs}/X)_{\mr{fl}}$. We have a satisfactory result in the important case that $G$ is a torus (as well as in a slightly more general case). The reason why can only deal with $R^i\varepsilon_{\mr{fl}*} G$ for $i=2$ and $G$ a suitable group scheme (mainly tori), is because we are only able to do computations of higher ($i>1$) group scheme cohomology in this case. 
 
The following proposition is the counterpart of Corollary \ref{cor1.2} in the Kummer flat topology, but only for $i=2$ and $F=G$ for a smooth commutative group scheme $G$.

\begin{prop}\label{prop1.4}
Let $X,x,A,p,P,X_m$, and $f_m$ be as in Corollary \ref{cor1.2}. Let $G$ be a smooth commutative group scheme over the underlying scheme of $X$.

Let $\gamma\in H_{\mr{kfl}}^2(X,G)$, then there exists a positive integer $n$ such that $\gamma$ maps to zero in $H_{\mr{kfl}}^2(X_n,G)$ along $f_n:X_n\rightarrow X$.
\end{prop}
\begin{proof}
By Proposition \ref{prop1.1}, we can find a Kummer flat cover $T\rightarrow X$ such that $\gamma$ dies in $H_{\mr{kfl}}^2(T,G)$. By \cite[Cor. 2.16]{niz1}, we may assume that for some $n$, we have a factorization $T\rightarrow X_n\rightarrow X$ with $T\rightarrow X_n$ a classical flat cover. It follows that the class $\gamma$ on $X_n$ is trivialized by a classical flat cover, i.e. $\gamma\in\mr{ker}(H_{\mr{kfl}}^2(X_n,G)\rightarrow H_{\mr{fl}}^0(X_n,R^2\varepsilon_{\mr{fl}*} G))$. The 7-term exact sequence of the spectral sequence $H_{\mr{fl}}^i(X_n,R^j\varepsilon_{\mr{fl}*}G)\Rightarrow H_{\mr{kfl}}^{i+j}(X_n,G)$ gives an exact sequence
$$\cdots\rightarrow H_{\mr{fl}}^2(X_n,G)\rightarrow \mr{ker}(H_{\mr{kfl}}^2(X_n,G)\rightarrow H_{\mr{fl}}^0(X_n,R^2\varepsilon_{\mr{fl}*} G))\rightarrow H_{\mr{fl}}^1(X_n,R^1\varepsilon_{\mr{fl}*} G).$$
Hence to show that $\gamma=0$ in $H_{\mr{kfl}}^2(X_n,G)$, it suffices to show that $H_{\mr{fl}}^2(X_n,G)=H_{\mr{fl}}^1(X_n,R^1\varepsilon_{\mr{fl}*} G)=0$.

Since $G$ is smooth and $A$ is strictly Henselian, we have 
$$H_{\mr{fl}}^2(X_n,G)=H_{\mr{\acute{e}t}}^2(X_n,G)=0.$$

We have $R^1\varepsilon_{\mr{fl}*}G=\varinjlim_m\mc{H}om_{X_n}(\Z/m\Z(1),G)\otimes (\Gml/\Gm)_{X_{n,\mr{fl}}}$ by Theorem \ref{thm1.5}. By Lemma \ref{lemA.1}, the sheaf $\mc{H}om_{X_n}(\Z/m\Z(1),G)$ is representable by a quasi-finite \'etale separated group scheme over $X_n$. It follows that 
\begin{align*}
H_{\mr{fl}}^1(X_n,R^1\varepsilon_{\mr{fl}*} G)=&H_{\mr{fl}}^1(X_n,\varinjlim_m\mc{H}om_{X_n}(\Z/m\Z(1),G)\otimes (\Gml/\Gm)_{X_{n,\mr{fl}}})  \\
=&\varinjlim_m H_{\mr{fl}}^1(X_n,\mc{H}om_{X_n}(\Z/m\Z(1),G)\otimes (\Gml/\Gm)_{X_{n,\mr{fl}}})  \\
=&\varinjlim_m H_{\mr{fl}}^1(X_n,\theta^*(\mc{H}om_{X_n}(\Z/m\Z(1),G)\otimes_{\Z}(\Gml/\Gm)_{X_{n,\mr{\acute{e}t}}}))  \\
=&\varinjlim_m H_{\mr{\acute{e}t}}^1(X_n,\mc{H}om_{X_n}(\Z/m\Z(1),G)\otimes_{\Z}(\Gml/\Gm)_{X_{n,\mr{\acute{e}t}}}) \\
=&0,
\end{align*}
where $\theta:(\mr{fs}/X_n)_{\mr{fl}}\rightarrow (\mr{fs}/X_n)_{\mr{\acute{e}t}}$ denotes the forgetful map between these two sites. This finishes the proof.
\end{proof}

The following corollary is the counterpart of Corollary \ref{cor1.3} and Proposition \ref{prop1.2} (2) for the Kummer flat topology.

\begin{cor}\label{cor1.8}
Let the notation and the assumptions be as in Proposition \ref{prop1.4}. Then we have
\begin{enumerate}[(1)]
\item $H_{\mr{kfl}}^2(X,G)\cong\varinjlim_n\mr{ker}(H_{\mr{kfl}}^2(X,G)\rightarrow H_{\mr{kfl}}^2(X_n,G))$;
\item $\check{H}_{\mr{kfl}}^0(\mathscr{X}_{\N},\underline{H}_{\mr{kfl}}^2(G))=0$.
\end{enumerate}
\end{cor}
\begin{proof}
This follows from Proposition \ref{prop1.4}.
\end{proof}

Let the notation and the assumptions be as in Proposition \ref{prop1.4}. Let $n$ be a positive integer. With the help of Corollary \ref{cor1.8} (2), the Kummer flat \v{C}ech cohomology to derive functor cohomology spectral sequence 
from (\ref{eq1.1}) gives rise to an exact sequence
\begin{equation}\label{eq1.3}
\begin{split}
0\rightarrow &\varinjlim_n\check{H}_{\mr{kfl}}^1(X_n/X,G)\rightarrow H_{\mr{kfl}}^1(X,G)\rightarrow \varinjlim_n\check{H}_{\mr{kfl}}^0(X_n/X,\underline{H}_{\mr{kfl}}^1(G))  \\
\rightarrow &\varinjlim_n\check{H}_{\mr{kfl}}^2(X_n/X,G)\rightarrow H_{\mr{kfl}}^2(X,G) \rightarrow \varinjlim_n\check{H}_{\mr{kfl}}^1(X_n/X,\underline{H}_{\mr{kfl}}^1(G))\\ 
\rightarrow &\varinjlim_n\check{H}_{\mr{kfl}}^3(X_n/X,G).
\end{split}
\end{equation}


\begin{lem}\label{lem1.5}
Let the notation and the assumptions be as in Proposition \ref{prop1.4}, then we have 
$$\varinjlim_{n}\check{H}_{\mr{kfl}}^i(X_n/X,\underline{H}_{\mr{kfl}}^1(G))=0$$
for $i\geq 0$.
\end{lem}
\begin{proof}
We denote by $X_{n,i}$ the fibre product of $i+1$ copies of $X_n$ over $X$. We have $X_{n,i}=H_n^i\times_{\Spec\Z}X_n=(H_n)_{X}^i\times_XX_n$, where $H_n:=\Spec\Z[(P^{1/n})^{\mr{gp}}/P^{\mr{gp}}]$ and $(H_n)_X:=H_n\times_{\Spec\Z}X$. Note that $(H_n)_X$ is a constant group scheme associated to the abstract group $H_n(X)$ over $X$ if $(n,p)=1$.

We first compute the \v{C}ech cohomology group $\check{H}_{\mr{kfl}}^i(X_n/X,\underline{H}_{\mr{kfl}}^1(G))$ for $n=p^r$. In this case, the underlying scheme of the group scheme $(H_{p^r})_{X}$
is strictly Henselian local. Let $A_{p^r}$ (resp. $x_{p^r}$) be the underlying ring (resp. the closed point) of $X_{p^r}$, and let $\mathfrak{m}_{A_{p^r}}$ be the maximal ideal of $A_{p^r}$. Since $A_{p^r}$ is a finite local $A$-algebra, it is also Henselian local by part (2) of \cite[\href{https://stacks.math.columbia.edu/tag/04GH}{Tag 04GH}]{stacks-project}, and thus $(A_{p^r},\mathfrak{m}_{A_{p^r}})$ is a Henselian pair (see \cite[\href{https://stacks.math.columbia.edu/tag/09XE}{Tag 09XE}]{stacks-project}). Let $B$ be the underlying ring of $(H_{p^r})_{X}^i\times_XX_{p^r}$. Clearly $B$ is a finite local $A_{p^r}$-algebra. Therefore any finite $B$-algebra $C$ is also a finite $A_{p^r}$-algebra, and we have $(\mathfrak{m}_{A_{p^r}}B)C=\mathfrak{m}_{A_{p^r}}C$. Then by an easy exercise, one can see that the equivalence $(1)\Leftrightarrow(3)$ of \cite[\href{https://stacks.math.columbia.edu/tag/09XI}{Tag 09XI}]{stacks-project} implies that $(B,\mathfrak{m}_{A_{p^r}}B)$ is also a Henselian pair. Since the Hom-sheaf $\mc{H}om_X(\Z/m\Z(1),G)$ is torsion, Gabber's Theorem (see \cite[\href{https://stacks.math.columbia.edu/tag/09ZI}{Tag 09ZI}]{stacks-project}) implies that the vertical maps in the canonical commutative diagram
$$\xymatrix{
\mr{Hom}_{X_{p^r}}(\Z/m\Z(1),G)\ar[r]\ar[d]^\cong &\mr{Hom}_{(H_{p^r})_{X}^i\times_XX_{p^r}}(\Z/m\Z(1),G)\ar[d]^\cong  \\ 
\mr{Hom}_{x_{p^r}}(\Z/m\Z(1),G)\ar[r]^-\cong &\mr{Hom}_{(H_{p^r})_{X}^i\times_Xx_{p^r}}(\Z/m\Z(1),G)
}$$
are isomorphisms. Since $(H_{p^r})_{X}^i\times_Xx_{p^r}$ is an infinitesimal thickening of $x_{p^r}$, the lower horizontal map in the above diagram is also an isomorphism by \cite[\href{https://stacks.math.columbia.edu/tag/03SI}{Tag 03SI}]{stacks-project}. Thus the canonical map 
$$\mr{Hom}_{X_{p^r}}(\Z/m\Z(1),G)\xrightarrow{\cong}\mr{Hom}_{(H_{p^r})_{X}^i\times_XX_{p^r}}(\Z/m\Z(1),G)$$
is an isomorphism. Combining with Theorem \ref{thm1.4}, we get 
\begin{align*}
H_{\mr{kfl}}^1(X_{p^r,i},G)=&H_{\mr{kfl}}^1((H_{p^r})_{X}^i\times_XX_{p^r},G) \\
=&\varinjlim_m\mr{Hom}_{(H_{p^r})_{X}^i\times_XX_{p^r}}(\Z/m\Z(1),G)\otimes_{\Z}(P^{\frac{1}{p^r}})^{\mr{gp}}  \\
=&\varinjlim_m\mr{Hom}_{X_{p^r}}(\Z/m\Z(1),G)\otimes_{\Z}(P^{\frac{1}{p^r}})^{\mr{gp}}  \\
=&H_{\mr{kfl}}^1(X_{p^r},G).
\end{align*}
The \v{C}ech complex for $\underline{H}_{\mr{kfl}}^1(G)$ with respect to the cover $X_{p^r}/X$ can be identified with 
$$H_{\mr{kfl}}^1(X_{p^r},G)\xrightarrow{0}H_{\mr{kfl}}^1(X_{p^r},G)\xrightarrow{\mr{Id}}H_{\mr{kfl}}^1(X_{p^r},G)\xrightarrow{0} H_{\mr{kfl}}^1(X_{p^r},G)\xrightarrow{\mr{Id}}\cdots ,$$
hence 
\begin{equation}\label{eq1.4}
\check{H}_{\mr{kfl}}^i(X_{p^r}/X,\underline{H}_{\mr{kfl}}^1(G))=\begin{cases}H_{\mr{kfl}}^1(X_{p^r},G),&\text{ if $i=0$}  \\ 0, &\text{ if $i>0$}\end{cases}.
\end{equation}

In general, we write $n=p^r\cdot n'$ with $(p,n')=1$, then we have 
$$X_{n,i}=(H_{n'})_{X}^i\times_X((H_{p^r})_{X}^i\times_XX_n)$$
and $$H_{\mr{kfl}}^1(X_{n,i},G)=\prod_{x\in H_{n'}(X)^i}H_{\mr{kfl}}^1(X_{n},G)=\mr{Map}(H_{n'}(X)^i,H_{\mr{kfl}}^1(X_{n},G)).$$
The \v{C}ech complex for $\underline{H}_{\mr{kfl}}^1(G)$ with respect to the cover $X_{n}/X$ can be identified with the standard complex that computes the cohomology of $H_{\mr{kfl}}^1(X_{n},G)$ regarded as a trivial $H_{n'}(X)$-module, hence we get $$\check{H}_{\mr{kfl}}^i(X_{n}/X,\underline{H}_{\mr{kfl}}^1(G))=H^i(H_{n'}(X),H_{\mr{kfl}}^1(X_{n},G)).$$

At last, we get
\begin{align*}
\varinjlim_{n}\check{H}_{\mr{kfl}}^i(X_n/X,\underline{H}_{\mr{kfl}}^1(G))=&\varinjlim_{n=p^r\cdot n'} H^i(H_{n'}(X),H_{\mr{kfl}}^1(X_{n},G))  \\
=&H^i(\varprojlim_{n'}H_{n'}(X),\varinjlim_{n=p^r\cdot n'}H_{\mr{kfl}}^1(X_{n},G)),
\end{align*}
where the second identification follows from \cite[\S 2, Prop. 8]{ser1}. By Proposition \ref{prop1.3} we have $\varinjlim_{n}H_{\mr{kfl}}^1(X_{n},G)=0$, and thus $\varinjlim_{n}\check{H}_{\mr{kfl}}^i(X_n/X,\underline{H}_{\mr{kfl}}^1(G))=0$. 
\end{proof}

\begin{thm}\label{thm1.7}
Let the notation and the assumptions be as in Proposition \ref{prop1.4}. Then the canonical homomorphism $\varinjlim_{n}\check{H}_{\mr{kfl}}^2(X_n/X,G)\rightarrow H_{\mr{kfl}}^2(X,G)$ is an isomorphism.
\end{thm}
\begin{proof}
The result follows from Lemma \ref{lem1.5} and the exact sequence (\ref{eq1.3}).
\end{proof}

In order to understand the group $H_{\mr{kfl}}^2(X,G)$, we are reduced to compute the groups $\check{H}_{\mr{kfl}}^2(X_n/X,G)$. 

\begin{lem}\label{lem1.6}
Let the notation and the assumptions be as in Proposition \ref{prop1.4}. We further assume that $G$ is a torus. Then we have
\begin{enumerate}[(1)]
\item $\check{H}_{\mr{kfl}}^2(X_n/X,G)=\check{H}_{\mr{k\acute{e}t}}^2(X_n/X,G)$ for $(n,p)=1$;
\item $\check{H}_{\mr{kfl}}^2(X_{p^r}/X,G)=0$ for $r>0$;
\item $$H_{\mr{kfl}}^2(X,G)=\varinjlim_{(n,p)=1}\check{H}_{\mr{kfl}}^2(X_n/X,G)=\varinjlim_{(n,p)=1}\check{H}_{\mr{k\acute{e}t}}^2(X_n/X,G)=H_{\mr{k\acute{e}t}}^2(X,G),$$
in particular $H_{\mr{kfl}}^2(X,G)$ is torsion, $p$-torsion-free and divisible.
\end{enumerate}
\end{lem}
\begin{proof}
(1) Since $X_n$ is a Kummer \'etale cover of $X$ whenever $(p,n)=1$, this is clear.

(2) By Lemma \ref{lem1.3}, we may assume that $X$ is a log point $\Spec k$ with $k$ separably closed. By \cite[page 521-523, in particular Lem. 3.16]{niz1}, we have 
$$\check{H}_{\mr{kfl}}^i(X_{p^r}/X,G)\cong H^i_{X_{\mr{fl}}}(H_{p^r},G)$$
for $i\geq 1$, where $H^i_{X_{\mr{fl}}}(H_{p^r},G)$ denotes the $i$-th cohomology group of the group scheme $H_{p^r}=\Spec\Z[(P^{1/p^r})^{\mr{gp}}/P^{\mr{gp}}]$ acting trivially on $G$ over the flat site $X_{\mr{fl}}$. The group $H^2_{X_{\mr{fl}}}(H_{p^r},G)$ can be identified with the group of extension classes of $H_{p^r}$ by $G$ which admit a (not necessarily homomorphic) section, see \cite[Exp. XVII, Prop. A.3.1]{sga3-2}. By \cite[Exp. XVII, Prop. 7.1.1]{sga3-2}, such extensions must be of multiplicative type, therefore must be commutative. Since the base field $k$ is separably closed, such extensions must be trivial. It follows that $\check{H}_{\mr{kfl}}^2(X_{p^r}/X,G)\cong H^2_{X_{\mr{fl}}}(H_{p^r},G)=0$. 

(3) First we show that any class $\gamma\in H_{\mr{kfl}}^2(X,G)$ vanishes in $H_{\mr{kfl}}^2(X_m,G)$ for some positive integer $m$ with $(m,p)=1$. By Corollary \ref{cor1.8}, $\gamma$ is annihilated by some cover $X_{m\cdot p^r}$ with $(m,p)=1$. Let $\gamma'$ be the image of $\gamma$ in $H_{\mr{kfl}}^2(X_m,G)$, we want to show that $\gamma'$ is zero. The \v{C}ech-to-derived functor spectral sequence for the cover $X_{m\cdot p^r}/X_{m}$ gives rise to an exact sequence
\begin{align*}
\cdots\rightarrow &\check{H}_{\mr{kfl}}^2(X_{m\cdot p^r}/X_m,G)\rightarrow \mr{ker}(H_{\mr{kfl}}^2(X_m,G)\rightarrow H_{\mr{kfl}}^2(X_{m\cdot p^r},G)) \\
\rightarrow &\check{H}_{\mr{kfl}}^1(X_{m\cdot p^r}/X_m,\underline{H}_{\mr{kfl}}^1(G))\rightarrow\cdots .
\end{align*}
We have $\gamma'\in \mr{ker}(H_{\mr{kfl}}^2(X_m,G)\rightarrow H_{\mr{kfl}}^2(X_{m\cdot p^r},G))$. We have $\check{H}_{\mr{kfl}}^2(X_{m\cdot p^r}/X_m,G)=0$ by part (2). By (\ref{eq1.4}), we have $\check{H}_{\mr{kfl}}^1(X_{m\cdot p^r}/X_m,\underline{H}_{\mr{kfl}}^1(G))=0$. Hence we get $\gamma'=0$. It follows that
\begin{equation}\label{eq1.5}
\varinjlim_{(m,p)=1}\mr{ker}(H_{\mr{kfl}}^2(X,G)\rightarrow H_{\mr{kfl}}^2(X_m,G))\xrightarrow{\cong}H_{\mr{kfl}}^2(X,G).
\end{equation}

Now consider the exact sequences 
\begin{align*}
\rightarrow &\check{H}_{\mr{kfl}}^0(X_m/X,\underline{H}_{\mr{kfl}}^1(G))\rightarrow \check{H}_{\mr{kfl}}^2(X_{m}/X,G)\rightarrow \mr{ker}(H_{\mr{kfl}}^2(X,G)\rightarrow H_{\mr{kfl}}^2(X_m,G)) \\
\rightarrow &\check{H}_{\mr{kfl}}^1(X_m/X,\underline{H}_{\mr{kfl}}^1(G))\rightarrow
\end{align*}
arising from the spectral sequence $\check{H}_{\mr{kfl}}^i(X_m/X,\underline{H}_{\mr{kfl}}^j(G))\Rightarrow H_{\mr{kfl}}^{i+j}(X_m,G)$ for the Kummer flat covers $X_m/X$ with $(m,p)=1$. Taking direct limit, we get an exact sequence 
\begin{equation}\label{eq1.6}
\begin{split}
\rightarrow &\varinjlim_{(m,p)=1}\check{H}_{\mr{kfl}}^0(X_m/X,\underline{H}_{\mr{kfl}}^1(G))\xrightarrow{\alpha} \varinjlim_{(m,p)=1}\check{H}_{\mr{kfl}}^2(X_{m}/X,G)\rightarrow H_{\mr{kfl}}^2(X,G) \\  
\rightarrow &\varinjlim_{(m,p)=1}\check{H}_{\mr{kfl}}^1(X_m/X,\underline{H}_{\mr{kfl}}^1(G))
\end{split}
\end{equation}
by the identification (\ref{eq1.5}). Similar to the general case of the proof of Lemma \ref{lem1.5}, we can show that $\varinjlim_{(m,p)=1}\check{H}_{\mr{kfl}}^i(X_m/X,\underline{H}_{\mr{kfl}}^1(G))=0$ for any $i\geq0$. Then the exact sequence (\ref{eq1.6}) tells us that 
$$H_{\mr{kfl}}^2(X,G)=\varinjlim_{(m,p)=1}\check{H}_{\mr{kfl}}^2(X_m/X,G)=\varinjlim_{(m,p)=1}\check{H}_{\mr{k\acute{e}t}}^2(X_m/X,G)=H_{\mr{k\acute{e}t}}^2(X,G).$$
The group $H_{\mr{k\acute{e}t}}^2(X,G)$ is torsion and $p$-torsion-free by Corollary \ref{cor1.4}, and $n$-divisible for $(n,p)=1$ by Theorem \ref{thm1.1} (3). Therefore $H_{\mr{kfl}}^2(X,G)$ is torsion, $p$-torsion-free, and divisible.
\end{proof}

\begin{cor}\label{cor1.9}
Let the notation and the assumptions be as in Lemma \ref{lem1.6}. Then we have 
$H_{\mr{kfl}}^2(X,G[n])\cong G[m](-2)(X)\otimes_{\Z}\bigwedge^2P^{\mr{gp}}$, where $n=m\cdot p^r$ with $(m,p)=1$.
\end{cor}
\begin{proof}
We have a short exact sequence 
$$0\rightarrow H_{\mr{kfl}}^1(X,G)\otimes_{\Z}\Z/n\Z\rightarrow H_{\mr{kfl}}^2(X,G[n])\rightarrow H_{\mr{kfl}}^2(X,G)[n]\rightarrow 0.$$
Since $G$ is a torus, the group $H_{\mr{kfl}}^1(X,G)=\varinjlim_{n}\mr{Hom}_{X_{\mr{kfl}}}(\Z/n\Z(1),G)\otimes_{\Z}P^{\mr{gp}}$ is divisible and whence $H_{\mr{kfl}}^1(X,G)\otimes_{\Z}\Z/n\Z=0$. Therefore for $n=m\cdot p^r$ with $(m,p)=1$, we have
\begin{align*}
H_{\mr{kfl}}^2(X,G[n])=&H_{\mr{kfl}}^2(X,G)[n]=H_{\mr{k\acute{e}t}}^2(X,G)[n]=H_{\mr{k\acute{e}t}}^2(X,G)[m]  \\
=&H_{\mr{k\acute{e}t}}^2(X,G[m])=G[m](-2)(X)\otimes_{\Z}\bigwedge^2P^{\mr{gp}}.
\end{align*}
\end{proof}

\begin{cor}\label{cor1.10}
Let $X$ be a locally  noetherian fs log scheme, and let $G$ be a torus over the underlying scheme of $X$. Let $Y\in (\mr{fs}/X)$ be such that the ranks of the stalks of the \'etale sheaf $M_Y/\mc{O}_Y^{\times}$ are at most one, and let $(\mr{st}/Y)$ be the full subcategory of $(\mr{fs}/X)$ consisting of strict fs log schemes over $Y$. Then we have that the restriction of $R^2\varepsilon_{\mr{fl}*}G$ to $(\mr{st}/Y)$ is zero.
\end{cor}
\begin{proof}
This follows from Corollary \ref{cor1.9}.
\end{proof}

\begin{thm}\label{thm1.8}
Let $X$ be a locally  noetherian fs log scheme, and let $G$ be a torus over the underlying scheme of $X$. 
\begin{enumerate}[(1)]
\item We have 
$$R^2\varepsilon_{\mr{fl}*}G=\varinjlim_{n} (R^2\varepsilon_{\mr{fl}*}G)[n]=\bigoplus_{\text{$l$ prime}}(R^2\varepsilon_{\mr{fl}*}G)[l^\infty],$$
where $(R^2\varepsilon_{\mr{fl}*}G)[n]$ denotes the $n$-torsion subsheaf of $R^2\varepsilon_{\mr{fl}*}G$ and $(R^2\varepsilon_{\mr{fl}*}G)[l^\infty]$ denotes the $l$-primary part of $R^2\varepsilon_{\mr{fl}*}G$ for a prime number $l$.
\item We have $(R^2\varepsilon_{\mr{fl}*}G)[n]=R^2\varepsilon_{\mr{fl}*}G[n]$.
\item The $l$-primary part $(R^2\varepsilon_{\mr{fl}*}G)[l^\infty]$ is supported on the locus where $l$ is invertible.
\item If $n$ is invertible on $X$, then 
$$(R^2\varepsilon_{\mr{fl}*}G)[n]=R^2\varepsilon_{\mr{fl}*}G[n]=G[n](-2)\otimes_{\Z}\bigwedge^2(\Gml/\Gm)_{X_{\mr{fl}}}.$$
\end{enumerate}
\end{thm}
\begin{proof}
By Lemma \ref{lem1.6}, $R^2\varepsilon_{\mr{fl}*}G$ is torsion. Hence part (1) follows.

We prove part (2). Since $G$ is a torus, we have a short exact sequence $0\to G[n]\to G\xrightarrow{n}G\to0$ for any $n\geq1$. This short exact sequence induces a short exact sequence
$$0\rightarrow R^1\varepsilon_{\mr{fl}*}G\otimes_{\Z}\Z/n\Z\rightarrow R^2\varepsilon_{\mr{fl}*}G[n]\rightarrow (R^2\varepsilon_{\mr{fl}*}G)[n]\rightarrow 0.$$
Since the sheaf $R^1\varepsilon_{\mr{fl}*}G$ is divisible, we get $(R^2\varepsilon_{\mr{fl}*}G)[n]=R^2\varepsilon_{\mr{fl}*}G[n]$. 

Part (3) follows from part (2) and Corollary \ref{cor1.9}.

We are left with part (4).  By \cite[\S 3]{swa1}, in particular the part between \cite[Cor. 3.7]{swa1} and \cite[Thm. 3.8]{swa1}, we have cup-product for the higher direct image functors for the map of sites $\varepsilon_{\mr{fl}}:(\mr{fs}/X)_{\mr{kfl}}\rightarrow (\mr{fs}/X)_{\mr{fl}}$. The cup-product induces homomorphisms
$$G[n]\otimes_{\Z/n\Z}\bigwedge^2R^1\varepsilon_{\mr{fl}*}\Z/n\Z  
\rightarrow G[n]\otimes_{\Z/n\Z}R^2\varepsilon_{\mr{fl}*}\Z/n\Z\rightarrow R^2\varepsilon_{\mr{fl}*}G[n].$$
Since $G[n]\otimes_{\Z/n\Z}\bigwedge^2R^1\varepsilon_{\mr{fl}*}\Z/n\Z=G[n](-2)\otimes_{\Z}\bigwedge^2(\Gml/\Gm)_{X_{\mr{fl}}}$, we get a canonical homomorphism
$$G[n](-2)\otimes_{\Z}\bigwedge^2(\Gml/\Gm)_{X_{\mr{fl}}}\rightarrow R^2\varepsilon_{\mr{fl}*}G[n].$$
By Corollary \ref{cor1.9}, this homomorphism is an isomorphism. This finishes the proof of part (4).
\end{proof}

\begin{cor}\label{cor1.11}
Let $X$ be a locally noetherian fs log scheme such that the underlying scheme of $X$ is a $\Q$-scheme, and $G$ a torus over the underlying scheme of $X$. Then we have 
$$R^2\varepsilon_{\mr{fl}*}G=\varinjlim_{n}G[n](-2)\otimes_{\Z}\bigwedge^2(\Gml/\Gm)_{X_{\mr{fl}}}.$$
\end{cor}

\begin{cor}\label{cor1.12}
Let $p$ be a prime number. Let $X$ be a locally noetherian fs log scheme such that the underlying scheme of $X$ is an $\F_p$-scheme, and $G$ a torus over the underlying scheme of $X$. Then we have 
$$R^2\varepsilon_{\mr{fl}*}G=\varinjlim_{(n,p)=1}G[n](-2)\otimes_{\Z}\bigwedge^2(\Gml/\Gm)_{X_{\mr{fl}}}.$$
\end{cor}

\begin{thm}\label{thm1.9}
Let $X$ be a locally noetherian fs log scheme, and let $G$ be a smooth affine commutative group scheme over the underlying scheme of $X$.
\begin{enumerate}[(1)]
\item If the fibres $G_{x}$ of $G$ over $X$ are all unipotent and $\kappa(x)$-solvable (see \cite[Exp. XVII, Def. 5.1.0]{sga3-2}), then we have $$R^1\varepsilon_{\mr{fl}*}G=R^2\varepsilon_{\mr{fl}*}G=0.$$
\item If $G$ is an extension of a group $U$ by a torus $T$ such that the fibres $U_x$ of $U$ over $X$ are all unipotent and $\kappa(x)$-solvable, then the canonical map $R^2\varepsilon_{\mr{fl}*}T\rightarrow R^2\varepsilon_{\mr{fl}*}G$ is an isomorphism.
\end{enumerate}
\end{thm}
\begin{proof}
In part (2), since $G$ is smooth, $U$ has to be smooth by fppf descent. Hence part (2) follows from part (1). We are left to prove part (1).

By Theorem \ref{thm1.5}, we have 
$$R^1\varepsilon_{\mr{fl},*}G=\varinjlim_{n}\mc{H}om_{X}(\Z/n\Z(1),G)\otimes_{\Z}(\Gml/\Gm)_{X_{\mr{fl}}},$$
which is zero by \cite[Exp. XVII, Lem. 2.5]{sga3-2}. To prove $R^2\varepsilon_{\mr{fl},*}G=0$, it suffices to prove $H^2_{\mr{kfl}}(X,G)=0$ for the case that the underlying scheme of $X$ is $\Spec A$ with $A$ a strictly henselian ring and $X$ admits a chart $P\rightarrow M_X$ with $P\xrightarrow{\cong}M_{X,x}/\mc{O}_{X,x}^{\times}$. By Theorem \ref{thm1.7}, we have $\varinjlim_n\check{H}_{\mr{kfl}}^2(X_n/X,G)\xrightarrow{\cong} H_{\mr{kfl}}^2(X,G)$. By Lemma \ref{lem1.3}, we are further reduced to the case that $X$ is a log point with $A$ a separably closed field. We have that $\check{H}_{\mr{kfl}}^2(X_n/X,G)=H^2_{X_{\mr{fl}}}(H_n,G)$ by \cite[page 521-523, in particular Lem. 3.16]{niz1}. But $H^2_{X_{\mr{fl}}}(H_n,G)=0$ for any positive integer $n$ by \cite[Exp. XVII, Thm. 5.1.1 (1) (c)]{sga3-2} and \cite[Exp. XVII, App. I, Prop. 3.1]{sga3-2}. This finishes the proof.
\end{proof}

\section{The higher direct images of the logarithmic multiplicative group}\label{sec2}
In this section, we show that $R^i\varepsilon_{\mr{fl}*}\Gmlb=0$ (resp. $R^i\varepsilon_{\mr{\acute{e}t}*}\Gmlb=0$) for $i\geq 1$, where $\Gmlb$ denotes the quotient of $\Gml$ by $\Gm$ with respect to the Kummer flat topology (resp. Kummer \'etale topology). The case $i=1$ has been treated essentially in the proof of \cite[Cor. 3.21]{niz1}. As a corollary, we get $R^i\varepsilon_{\mr{fl}*}\Gm\cong R^i\varepsilon_{\mr{fl}*}\Gml$ (resp. $R^i\varepsilon_{\mr{\acute{e}t}*}\Gm\cong R^i\varepsilon_{\mr{\acute{e}t}*}\Gml$) for $i\geq 2$. By Kato's logarithmic Hilbert 90, see \cite[Cor. 3.21]{niz1}, we have $R^1\varepsilon_{\mr{fl*}}\Gml=0$ (resp. $R^1\varepsilon_{\mr{\acute{e}t}*}\Gml=0$).

We start with the strictly Henselian case.

\begin{thm}\label{thm2.1}
Let $X=\Spec A$ be an fs log scheme with $A$ a noetherian strictly Henselian local ring, $x$ the closed point of $X$, and $p$ the characteristic of the residue field of $A$. We fix a chart $P\rightarrow M_X$ satisfying $P\xrightarrow{\cong} M_{X,x}/\mc{O}_{X,x}^{\times}$.  Then we have 
\begin{enumerate}[(1)]
\item $H_{\mr{kfl}}^r(X,\Gmlb)=0$ (resp. $H_{\mr{k\acute{e}t}}^r(X,\Gmlb)=0$) for $r\geq 1$;
\item $H_{\mr{kfl}}^r(X,\Gm)\cong H_{\mr{kfl}}^r(X,\Gml)$ (resp. $H_{\mr{k\acute{e}t}}^r(X,\Gm)\cong H_{\mr{k\acute{e}t}}^r(X,\Gml)$) for $r\geq 2$.
\end{enumerate}
\end{thm}

Before going to the proof of Theorem \ref{thm2.1}, we prove the following lemma.

\begin{lem}\label{lem2.1}
Let the notation and the assumptions be as in Theorem \ref{thm2.1}. For the \v{C}ech cohomology for the Kummer flat cover $X_n/X$, we have 
$$\check{H}_{\mr{kfl}}^i(X_n/X,\Gmlb)=\begin{cases}
(P^{\frac{1}{n}})^{\mr{gp}}\otimes_{\Z}\Q, & \text{if $i=0$;} \\
0,  &\text{if $i>0$.}
\end{cases}
$$
If $(n,p)=1$, for the \v{C}ech cohomology for the Kummer \'etale cover $X_n/X$, we have 
$$\check{H}_{\mr{k\acute{e}t}}^i(X_n/X,\Gmlb)=\begin{cases}
(P^{\frac{1}{n}})^{\mr{gp}}\otimes_{\Z}\Q', & \text{if $i=0$;} \\
0,  &\text{if $i>0$.}
\end{cases}
$$
\end{lem}
\begin{proof}
We only deal with the Kummer flat case, the Kummer \'etale case can be done in the same way.

It is clear that 
\begin{equation}\label{eq2.1}
\check{H}_{\mr{kfl}}^0(X_n/X,\Gmlb)=(\Gml/\Gm)(X_n)=(P^{\frac{1}{n}})^{\mr{gp}}\otimes_{\Z}\Q.
\end{equation}
Let $n=m\cdot p^t$ with $(m,p)=1$, then $X\times_{\Spec\Z}H_m^r$ is a constant group scheme over $X$ and $X\times_{\Spec\Z}H_{p^t}^r$ is a connected group scheme over $X$, therefore we have
\begin{align*}
\Gmlb(\underbrace{X_n\times_X\cdots\times_XX_n}_{\text{$r+1$ times}})=&\Gmlb(X_n\times_{\Spec\Z}H_n^r) \\
=&\Gmlb((X_n\times_{\Spec\Z}H_{p^t}^r)\times_X(X\times_{\Spec\Z}H_m^r)) \\
=&\prod_{h\in H_m(X)^r}\Gmlb(X_n\times_{\Spec\Z}H_{p^t}^r)  \\
=&\prod_{h\in H_m(X)^r}\Gmlb(X_n)  \\
=&\prod_{h\in H_m(X)^r}(P^{\frac{1}{n}})^{\mr{gp}}\otimes_{\Z}\Q
\\
=&\mr{Map}(H_m(X)^r,(P^{\frac{1}{n}})^{\mr{gp}}\otimes_{\Z}\Q).
\end{align*}
To compute the higher \v{C}ech cohomology groups, we consider the \v{C}ech complex 
\begin{equation}\label{eq2.2}
\Gmlb(X_n)\xrightarrow{d_0}\Gmlb(X_n\times_XX_n)\xrightarrow{d_1}\Gmlb(X_n\times_XX_n\times_XX_n)\xrightarrow{d_2}\cdots
\end{equation}
for $\Gmlb$ with respect to the cover $X_n/X$. Let $\Gamma_n:=(P^{\frac{1}{n}})^{\mr{gp}}/P^{\mr{gp}}$. By \cite[Chap. III, Example 2.6]{mil1}, the \v{C}ech nerve of the Kummer flat cover $X_n/X$ can be identified with the sequence
$$\xymatrix{
X_n & X_n\times H_n\ar@<0.5ex>[l]^-{d_{1,0}}\ar@<-0.5ex>[l]_-{d_{1,1}} &X_n\times H_n^2\ar@<1ex>[l]^-{d_{2,0}}\ar@<0ex>[l]\ar@<-1ex>[l]_-{d_{2,2}} &X_n\times H_n^3\ar@<1.5ex>[l]^-{d_{3,0}}\ar@<0.5ex>[l]\ar@<-0.5ex>[l]\ar@<-1.5ex>[l]_-{d_{3,3}}\cdots ,
}$$
where the map $d_{r,i}$ on the ring level is given by the $A$-linear ring homomorphism
\begin{align*}
A\otimes_{\Z[P]}\Z[P^{\frac{1}{n}}\oplus\Gamma_n^{r-1}]&\rightarrow A\otimes_{\Z[P]}\Z[P^{\frac{1}{n}}\oplus\Gamma_n^{r}]
  \\
(a,\bar{a}_1,\cdots,\bar{a}_{r-1})&\mapsto
\begin{cases} (a,\bar{a},\bar{a}_1,\cdots,\bar{a}_{r-1}),&\text{if $i=0$;}  \\
(a,\bar{a}_1,\cdots,\bar{a}_{i},\bar{a}_i,\cdots,\bar{a}_{r-1}),&\text{if $0<i<r$;}  \\
 (a,\bar{a}_1,\cdots,\bar{a}_{r-1},0),&\text{if $i=r$.}
\end{cases}
\end{align*}
for any $(a,\bar{a}_1,\cdots,\bar{a}_{r-1})\in P^{\frac{1}{n}}\oplus\Gamma_n^{r-1}$. If $m=1$, i.e. $n=p^t$, we have 
$$\Gmlb(X_{p^t}\times_{\Spec\Z}H_{p^t}^r)=(P^{\frac{1}{p^t}})^{\mr{gp}}\otimes_{\Z}\Q.$$
By the description of $d_{r,i}$, the map 
$$d_{r,i}^*:\Gmlb(X_{p^t}\times_{\Spec\Z}H_{p^t}^{r-1})\rightarrow \Gmlb(X_{p^t}\times_{\Spec\Z}H_{p^t}^r)$$
can be identified with the identity map 
$$\mr{Id}:(P^{\frac{1}{p^t}})^{\mr{gp}}\otimes_{\Z}\Q\rightarrow (P^{\frac{1}{p^t}})^{\mr{gp}}\otimes_{\Z}\Q.$$
In general, the map 
$$d_{r,i}^*:\Gmlb(X_{n}\times_{\Spec\Z}H_{n}^{r-1})\rightarrow \Gmlb(X_{n}\times_{\Spec\Z}H_{n}^r)$$
can be identified with the map
\begin{align*}
\mr{Map}(H_m(X)^{r-1},(P^{\frac{1}{n}})^{\mr{gp}}\otimes_{\Z}\Q)\xrightarrow{\partial_{r,i}} \mr{Map}(H_m(X)^r,(P^{\frac{1}{n}})^{\mr{gp}}\otimes_{\Z}\Q)  \\
f\mapsto (\partial_{r,i}(f):(h_1,\cdots,h_r)\mapsto 
\begin{cases}
f(h_2,\cdots,\cdots,h_r), &\text{if $i=0$} \\
f(h_1,\cdots,h_i+h_{i+1},\cdots,h_r), &\text{if $0<i<r$} \\
f(h_1,\cdots\cdots,h_{r-1}), &\text{if $i=r$} 
\end{cases})
\end{align*}
Therefore the complex (\ref{eq2.2}) can be identified with the standard complex
$$(P^{\frac{1}{n}})^{\mr{gp}}\otimes_{\Z}\Q\rightarrow \mr{Map}(H_m(X),(P^{\frac{1}{n}})^{\mr{gp}}\otimes_{\Z}\Q)\rightarrow \mr{Map}(H_m(X)^2,(P^{\frac{1}{n}})^{\mr{gp}}\otimes_{\Z}\Q)\rightarrow\cdots$$
for computing the group cohomology of the trivial $H_m(X)$-module $(P^{\frac{1}{n}})^{\mr{gp}}\otimes_{\Z}\Q$. It follows that 
$$\check{H}_{\mr{kfl}}^i(X_n/X,\Gmlb)\cong H^i(H_m(X),(P^{\frac{1}{n}})^{\mr{gp}}\otimes_{\Z}\Q)=\begin{cases}
(P^{\frac{1}{n}})^{\mr{gp}}\otimes_{\Z}\Q, & \text{if $i=0$;} \\
0,  &\text{if $i>0$.}
\end{cases}$$
\end{proof}

\begin{proof}[Proof of Theorem \ref{thm2.1}:]
Part (2) follows from part (1) clearly. We only deal with the Kummer flat case, the Kummer \'etale case can be proven by the same way.

We use induction on $r$ to prove part (1). 

First of all we consider the case $r=1$. We want to prove $H_{\mr{kfl}}^1(X,\Gmlb)=0$. The spectral sequence $\check{H}_{\mr{kfl}}^i(X_n/X,\underline{H}_{\mr{kfl}}^j(\Gmlb))\Rightarrow H^{i+j}_{\mr{kfl}}(X,\Gmlb)$ gives rise to an exact sequence
$$0\rightarrow \check{H}_{\mr{kfl}}^1(X_n/X,\Gmlb)\rightarrow H_{\mr{kfl}}^1(X,\Gmlb)\xrightarrow{u} \check{H}_{\mr{kfl}}^0(X_n/X,\underline{H}_{\mr{kfl}}^1(\Gmlb)).$$
By Lemma \ref{lem2.1}, we have that $\check{H}_{\mr{kfl}}^1(X_n/X,\Gmlb)$ vanishes, and thus $u$ is injective. We also have a canonical injection 
$$v:\check{H}_{\mr{kfl}}^0(X_n/X,\underline{H}_{\mr{fl}}^1(\Gmlb))\hookrightarrow H_{\mr{kfl}}^1(X_n,\Gmlb).$$
The composition $v\circ u$ is nothing but the pullback map $f_n^*:H_{\mr{kfl}}^1(X,\Gmlb)\rightarrow H_{\mr{kfl}}^1(X_n,\Gmlb)$, where $f_n$ denotes the cover map $X_n\rightarrow X$. Hence we get an injection $f_n^*:H_{\mr{kfl}}^1(X,\Gmlb)\hookrightarrow H_{\mr{kfl}}^1(X_n,\Gmlb)$. Passing to the direct limit, we get a canonical injection $H_{\mr{kfl}}^1(X,\Gmlb)\hookrightarrow \varinjlim_{n}H_{\mr{kfl}}^1(X_n,\Gmlb)$. Hence it suffices to show that $\varinjlim_{n}H_{\mr{kfl}}^1(X_n,\Gmlb)=0$. Let $\alpha$ be an element of $H_{\mr{kfl}}^1(X_n,\Gmlb)$, and let $T\rightarrow X_n$ be a Kummer flat cover such that $\alpha$ dies in $H_{\mr{kfl}}^1(T,\Gmlb)$. By \cite[Cor. 2.16]{niz1}, we may assume that for some $m$, we have a factorization $T\rightarrow X_{mn}\rightarrow X_n$, where $T\rightarrow X_{mn}$ is a classical flat cover. It follows that the class $\alpha$ on $X_{mn}$ is trivialized by a classical flat cover, i.e. $\alpha$ is mapped to zero along the map $H_{\mr{kfl}}^1(X_{mn},\Gmlb)\rightarrow H_{\mr{fl}}^0(X_{mn},R^1\varepsilon_{\mr{fl}*}\Gmlb)$. Hence $\alpha$ lies in the image of the map $H_{\mr{fl}}^1(X_{mn},\varepsilon_{\mr{fl}*}\Gmlb)\hookrightarrow H_{k\mr{fl}}^1(X_{mn},\Gmlb)$. But 
\begin{equation}\label{eq2.3}
\begin{split}
H_{\mr{fl}}^1(X_{mn},\varepsilon_{\mr{fl}*}\Gmlb)=&H_{\mr{fl}}^1(X_{mn},(\Gml/\Gm)_{X_{\mr{fl}}}\otimes_{\Z}\Q)  \\ =&H_{\mr{\acute{e}t}}^1(X_{mn},(\Gml/\Gm)_{X_{\mr{\acute{e}t}}}\otimes_{\Z}\Q)   \\
=&0.
\end{split}
\end{equation}
This finishes the proof of $H_{\mr{kfl}}^1(X,\Gmlb)=0$.

Now we fix a positive integer $r_0$, and assume that $H_{\mr{kfl}}^r(Y,\Gmlb)=0$ for all $0<r\leq r_0$ and all fs log schemes $Y$ satisfying the conditions for $X$ in the statement of the theorem. Note that this assumption implies that $R^r\varepsilon_{\mr{fl}*}\Gmlb=0$ for $0<r\leq r_0$. We are going to prove 
$$H_{\mr{kfl}}^{r_0+1}(X,\Gmlb)=0.$$ in two steps.

In the first step, we prove that the canonical map
\begin{equation}\label{eq2.4}
H^{r_0+1}_{\mr{kfl}}(X,\Gmlb)\rightarrow H^{r_0+1}_{\mr{kfl}}(X_n,\Gmlb)
\end{equation}
is injective for any $n>0$. Clearly $X_n$ satisfies the conditions for $X$ in the statement. For any $0<j\leq r_0$ and any $i\geq 0$, consider the $i$-th \v{C}ech cohomology group $\check{H}_{\mr{kfl}}^i(X_n/X,\underline{H}_{\mr{kfl}}^j(\Gmlb))$ of the Kummer flat cover $X_n/X$ with coefficients in the presheaf $\underline{H}_{\mr{kfl}}^j(\Gmlb)$. Since 
$$\underbrace{X_n\times_X\cdots\times_XX_n}_{\text{$k+1$ times}}=X_n\times_{\Spec\Z}H_n^k=(X_n\times_{\Spec\Z}H_{p^{t}}^k)\times_{X}(X\times_{\Spec\Z}H_{n'}^k)$$
with $n=n'\cdot p^t$ and $(n',p)=1$, $X_n\times_{\Spec\Z}H_{p^{t}}^k$ satisfies the conditions for $X$ in the statement, and $X\times_{\Spec\Z}H_{n'}^k$ is a constant group scheme over $X$, so we get 
$$\Gamma(\underbrace{X_n\times_X\cdots\times_XX_n}_{\text{$k+1$ times}},\underline{H}_{\mr{kfl}}^j(\Gmlb))=H_{\mr{kfl}}^j(\underbrace{X_n\times_X\cdots\times_XX_n}_{\text{$k+1$ times}},\Gmlb)=0.$$
It follows that $\check{H}_{\mr{kfl}}^i(X_n/X,\underline{H}_{\mr{kfl}}^j(\Gmlb))=0$ for any $0<j\leq r_0$ and any $i\geq 0$.
Then the spectral sequence 
$$\check{H}_{\mr{kfl}}^i(X_n/X,\underline{H}_{\mr{kfl}}^j(\Gmlb))\Rightarrow H^{i+j}_{\mr{kfl}}(X,\Gmlb)$$
implies that 
$$H^{r_0+1}_{\mr{kfl}}(X,\Gmlb)\xrightarrow{\cong} \check{H}_{\mr{kfl}}^0(X_n/X,\underline{H}_{\mr{kfl}}^{r_0+1}(\Gmlb)).$$
It follows that the canonical map
(\ref{eq2.4}) is injective for any $n>0$. 

In the second step, we finish the proof of $H_{\mr{kfl}}^{r_0+1}(X,\Gmlb)=0$. Let $\beta$ be any element of $H^{r_0+1}_{\mr{kfl}}(X,\Gmlb)$, and let $T\rightarrow X$ be a Kummer flat cover such that $\beta$ dies in $H^{r_0+1}_{\mr{kfl}}(T,\Gmlb)$. By \cite[Cor. 2.16]{niz1}, we may assume that for some $m$, we have a factorization $T\rightarrow X_m\rightarrow X$ such that $T\rightarrow X_m$ is a classical flat cover. It follows that the class $\beta$ on $X_m$ is trivialized by a classical flat cover, i.e. it lies in the kernel of the canonical map 
$$H^{r_0+1}_{\mr{kfl}}(X_m,\Gmlb)\rightarrow H^0_{\mr{fl}}(X_m,R^{r_0+1}\varepsilon_{\mr{fl}*}\Gmlb).$$
Consider the spectral sequence 
$$H^i_{\mr{fl}}(X_m,R^j\varepsilon_{\mr{fl}*}\Gmlb)\Rightarrow H^{i+j}_{\mr{kfl}}(X_m,\Gmlb).$$
The vanishing of $R^j\varepsilon_{\mr{fl}*}\Gmlb$ for $0<j\leq r_0$ gives rise to an exact sequence
$$0\rightarrow H_{\mr{fl}}^{r_0+1}(X_m,\varepsilon_{\mr{fl}*}\Gmlb)\rightarrow H_{\mr{kfl}}^{r_0+1}(X_m,\Gmlb)\rightarrow H^0_{\mr{fl}}(X_m,R^{r_0+1}\varepsilon_{\mr{fl}*}\Gmlb).$$
Hence the class $\beta$ on $X_m$ comes from $H_{\mr{fl}}^{r_0+1}(X_m,\varepsilon_{\mr{fl}*}\Gmlb)$ which is zero by the same reason as in (\ref{eq2.3}). Hence the class $\beta$ on $X_m$ is zero. Thus $\beta=0$ in $H_{\mr{kfl}}^{r_0+1}(X,\Gmlb)$ by the injectivity of the map (\ref{eq2.4}). This finishes the proof of $H_{\mr{kfl}}^{r_0+1}(X,\Gmlb)=0$.
\end{proof}

\begin{thm}
Let $X$ be an fs log scheme with its underlying scheme locally  noetherian. Then we have:
\begin{enumerate}[(1)]
\item $R^r\varepsilon_{\mr{fl}*}\Gmlb=0$ (resp. $R^r\varepsilon_{\mr{\acute{e}t}*}\Gmlb=0$) for $r\geq 1$;
\item the canonical map $R^r\varepsilon_{\mr{fl}*}\Gm\rightarrow R^r\varepsilon_{\mr{fl}*}\Gml$ (resp. $R^r\varepsilon_{\mr{\acute{e}t}*}\Gm\rightarrow R^r\varepsilon_{\mr{\acute{e}t}*}\Gml$) is an isomorphism for $r\geq 2$.
\end{enumerate}
\end{thm}

\begin{cor}
Let $X$ be a locally noetherian fs log scheme such that the stalks of $M_X^{\mr{gp}}/\mc{O}_X^{\times}$ for the classical \'etale topology have rank at most 1. Let $(\mr{k\acute{e}t}/X)_{\mr{k\acute{e}t}}$ (resp. $(\mr{k\acute{e}t}/X)_{\mr{\acute{e}t}}$) be the category of Kummer \'etale fs log schemes over $X$ endowed with the Kummer \'etale topology (resp. the classical \'etale topology), and let $\varepsilon:(\mr{k\acute{e}t}/X)_{\mr{k\acute{e}t}}\rightarrow (\mr{k\acute{e}t}/X)_{\mr{\acute{e}t}}$ be the canonical forgetful map of sites. Then we have $R^2\varepsilon_* \Gml=0$.
\end{cor}

\section{Examples}\label{sec3}
\subsection{Discrete valuation rings}
Let $R$ be a discrete valuation ring with fraction field $K$ and residue field $k$. Let $\pi$ be a uniformizer of $R$, and we endow $X=\Spec R$ with the log structure associated to the homomorphism $\N\rightarrow R,1\mapsto\pi$. Let $x$ be the closed point of $X$ and $i$ the closed immersion $x\hookrightarrow X$, and we endow $x$ with the induced log structure from $X$. Let $\eta$ be the generic point of $X$ and $j$ the open immersion $\eta\hookrightarrow X$.

Now we consider the Leray spectral sequence
\begin{equation}\label{eq3.1}
H^s_{\mr{fl}}(X,R^t\varepsilon_{\mr{fl}*}\Gm)\Rightarrow H^{s+t}_{\mr{kfl}}(X,\Gm).
\end{equation}
We have $H^s_{\mr{fl}}(X,\Gm)=H^s_{\mr{\acute{e}t}}(X,\Gm)$ for $s\geq 0$ by \cite[Thm. 11.7]{gro3}. We have 
\begin{align*}
R^1\varepsilon_{\mr{fl}*}\Gm=&\varinjlim_n\mc{H}om_X(\Z/n\Z(1),\Gm)\otimes_{\Z}(\Gml/\Gm)_{X_{\mr{fl}}}  \\
=&\Q/\Z\otimes_{\Z}(\Gml/\Gm)_{X_{\mr{fl}}}
\end{align*}
by Theorem \ref{thm1.5}.
Then on $(\mr{st}/X)$, we have $(\Gml/\Gm)_{X_{\mr{fl}}}\cong i_*\Z$. Therefore
\begin{equation}\label{eq3.2}
\begin{split}
H^s_{\mr{fl}}(X,R^1\varepsilon_{\mr{fl}*}\Gm)=H^s_{\mr{fl}}(X,\Q/\Z\otimes_{\Z}i_*\Z) =&H^s_{\mr{fl}}(x,\Q/\Z)  \\
=&H^s_{\mr{\acute{e}t}}(x,\Q/\Z)
\end{split}
\end{equation}
for $s\geq0$. We also have 
\begin{equation}\label{eq3.3}
H^s_{\mr{fl}}(X,R^2\varepsilon_{\mr{fl}*}\Gm)=H^s_{\mr{fl}}(X,0)=0
\end{equation}
for $s\geq0$ by Corollary \ref{cor1.10}. 

\begin{thm}
Assume that $k$ is a finite field. Then we have 
$$H^1_{\mr{kfl}}(X,\Gm)\cong H^0_{\mr{\acute{e}t}}(x,\Q/\Z)\cong\Q/\Z$$
and
$$H^2_{\mr{kfl}}(X,\Gm)\cong H^1_{\mr{\acute{e}t}}(x,\Q/\Z)\cong\Q/\Z.$$
\end{thm}
\begin{proof}
We have 
$$H^s_{\mr{fl}}(X,\Gm)=H^s_{\mr{\acute{e}t}}(X,\Gm)=0$$
for $s>0$ by \cite[Chap. II, Prop. 1.5 (a)]{mil2}. Since $k$ is a finite field, it has absolute Galois group $\hat{\Z}$. Therefore we can identify the group $H^1_{\mr{\acute{e}t}}(x,\Q/\Z)$ with the Galois cohomology 
$$H^1(\hat{\Z},\Q/\Z)=\mr{Hom}(\hat{\Z},\Q/\Z)\cong\Q/\Z.$$
Then the results follow from the spectral sequence (\ref{eq3.1}) with the help of (\ref{eq3.2}) and (\ref{eq3.3}).
\end{proof}

\subsection{Global Dedekind domains}
Through this subsection, let $K$ be a global field. When $K$ is a number field, $X$ denotes the spectrum of the ring of integers in $K$, and when $K$ is a function field, $k$ denotes the field of constants of $K$ and $X$ denotes the unique connected smooth projective curve over $k$ having $K$ as its function field. Let $S$ be a finite set of closed points of $X$, $U:=X-S$, $j:U\hookrightarrow X$, and $i_x:x\hookrightarrow X$ for each closed point $x\in X$. We endow $X$ with log structure $j_{*}\mc{O}^{\times}_U\cap\mc{O}_X\rightarrow \mc{O}_X$. In the case of a number field, let $S_{\infty}:=S\cup\{\text{infinite places of $K$}\}$, and in the case of function field, we just let $S_{\infty}:=S$.

On $(\mr{st}/X)$, we have $R^1\varepsilon_{\mr{fl}*}\Gm=\bigoplus_{x\in S}i_{x,*}\Q/\Z$ and $R^2\varepsilon_{\mr{fl}*}\Gm=0$. The Leray spectral sequence
$$H^s_{\mr{fl}}(X,R^t\varepsilon_{\mr{fl}*}\Gm)\Rightarrow H^{s+t}_{\mr{kfl}}(X,\Gm)$$ gives rise to a long exact sequence
\begin{equation}\label{eq3.4}
\begin{split}
0\rightarrow &H^1_{\mr{fl}}(X,\Gm)\rightarrow H^1_{\mr{kfl}}(X,\Gm)\xrightarrow{\alpha} H^0_{\mr{fl}}(X,R^1\varepsilon_{\mr{fl}*}\Gm)  \\
\rightarrow &H^2_{\mr{fl}}(X,\Gm)\rightarrow H^2_{\mr{kfl}}(X,\Gm)\rightarrow H^1_{\mr{fl}}(X,R^1\varepsilon_{\mr{fl}*}\Gm)  \\
\rightarrow &H^3_{\mr{fl}}(X,\Gm)\rightarrow H^3_{\mr{kfl}}(X,\Gm).
\end{split}
\end{equation}
The Leray spectral sequences for $\Gm$ and $\Gml$ respectively together give rise to the following commutative diagram
$$\xymatrix{
&0\ar[d] &0\ar[d]  \\
&H^1_{\mr{fl}}(X,\Gm)\ar[r]\ar[d] &H^1_{\mr{fl}}(X,\Gml)\ar[d]  \\
&H^1_{\mr{kfl}}(X,\Gm)\ar[r]\ar[d]^{\alpha} &H^1_{\mr{kfl}}(X,\Gml)\ar[d]  \\
&H^0_{\mr{fl}}(X,R^1\varepsilon_{\mr{fl}*}\Gm)\ar[r]\ar[d] &H^0_{\mr{fl}}(X,R^1\varepsilon_{\mr{fl}*}\Gml)\ar[d]  \\
H^1_{\mr{fl}}(X,(\Gml/\Gm)_{X_{\mr{fl}}})\ar[r] &H^2_{\mr{fl}}(X,\Gm)\ar[r] &H^2_{\mr{fl}}(X,\Gml)  \\
}$$
with exact rows and columns. We have $R^1\varepsilon_{\mr{fl}*}\Gml=0$ by Kato's logarithmic Hilbert 90 (see \cite[Cor. 3.21]{niz1}), and 
$$H^1_{\mr{fl}}(X,(\Gml/\Gm)_{X_{\mr{fl}}})=H^1_{\mr{fl}}(X,\bigoplus_{x\in S}i_{x,*}\Z)=\bigoplus_{x\in S}H^1_{\mr{fl}}(x,\Z)=0.$$ By diagram chasing, we find that the map $\alpha$ is surjective. We have 
$$H^0_{\mr{fl}}(X,R^1\varepsilon_{\mr{fl}*}\Gm)=\bigoplus_{x\in S}H^0_{\mr{fl}}(X,i_{x,*}\Q/\Z)=\bigoplus_{x\in S}H^0_{\mr{fl}}(x,\Q/\Z)=\bigoplus_{x\in S}\Q/\Z$$
and
$$H^1_{\mr{fl}}(X,R^1\varepsilon_{\mr{fl}*}\Gm)=\bigoplus_{x\in S}H^1_{\mr{fl}}(X,i_{x,*}\Q/\Z)=\bigoplus_{x\in S}H^1_{\mr{fl}}(x,\Q/\Z).$$
We also have 
$$H^2_{\mr{fl}}(X,\Gm)\cong\begin{cases}
(\Z/2\Z)^{r-1},&\text{ if $K$ has $r>0$ real places;} \\
0,&\text{ otherwise}
\end{cases}$$
and $H^3_{\mr{fl}}(X,\Gm)\cong\Q/\Z$ by \cite[Chap. II, Prop. 2.1]{mil2}. Therefore the exact sequence (\ref{eq3.4}) splits into two exact sequences
\begin{equation}
0\rightarrow H^1_{\mr{fl}}(X,\Gm)\rightarrow H^1_{\mr{kfl}}(X,\Gm)\xrightarrow{\alpha} H^0_{\mr{fl}}(X,R^1\varepsilon_{\mr{fl}*}\Gm)\rightarrow 0  
\end{equation}
and
\begin{equation}
0\rightarrow H^2_{\mr{fl}}(X,\Gm)\rightarrow H^2_{\mr{kfl}}(X,\Gm)\rightarrow H^1_{\mr{fl}}(X,R^1\varepsilon_{\mr{fl}*}\Gm)\rightarrow H^3_{\mr{fl}}(X,\Gm),
\end{equation}
which can be identified with 
\begin{equation}\label{eq3.7}
0\rightarrow H^1_{\mr{fl}}(X,\Gm)\rightarrow H^1_{\mr{kfl}}(X,\Gm)\xrightarrow{\alpha} \bigoplus_{x\in S}\Q/\Z\rightarrow 0  
\end{equation}
and
\begin{equation}
0\rightarrow H^2_{\mr{fl}}(X,\Gm)\rightarrow H^2_{\mr{kfl}}(X,\Gm)\rightarrow \bigoplus_{x\in S}H^1_{\mr{fl}}(x,\Q/\Z) \rightarrow \Q/\Z
\end{equation}
with $H^2_{\mr{fl}}(X,\Gm)=\begin{cases}
(\Z/2\Z)^{r-1}, &\text{ if $K$ has $r>0$ real palces;}   \\
0, &\text{ otherwise.}
\end{cases}$

Let $\mr{Pic}(X)$ (resp. $\mr{Pic}(X^{\mr{log}})$) denote the group $H^1_{\mr{fl}}(X,\Gm)$ (resp. $H^1_{\mr{kfl}}(X,\Gm)$). We have a canonical degree map $\mr{deg}:\mr{Pic}(X)\rightarrow\Z$ in both the number field case and the function field case. By the short exact sequence (\ref{eq3.7}), the degree map $\mr{deg}:\mr{Pic}(X)\rightarrow\Z$ extends uniquely to a map 
\begin{equation}
\mr{deg}:\mr{Pic}(X^{\mr{log}})\rightarrow\Q,
\end{equation}
and the two degree maps fit into the following commutative diagram
\begin{equation}
\xymatrix{
0\ar[r] &\mr{Pic}(X)\ar[r]\ar[d]^{\mr{deg}} &\mr{Pic}(X^{\mr{log}})\ar[r]\ar[d]^{\mr{deg}} &\bigoplus_{x\in S}\Q/\Z\ar[r]\ar[d]^{\mr{sum}} &0 \\
0\ar[r] &\Z\ar[r] &\Q\ar[r] &\Q/\Z\ar[r] &0
}
\end{equation}
with exact rows. 

To summarise, we get the following proposition.

\begin{prop}
Let the notation and the assumptions be as in the beginning of this subsection. Then we have the following.
\begin{enumerate}[(1)]
\item The group $\mr{Pic}(X^{\mr{log}}):=H^1_{\mr{kfl}}(X,\Gm)$ admits a canonical degree map into $\Q$ which extends the canonical degree map on $\mr{Pic}(X)$, and the two degree maps fit into the following commutative digram
$$\xymatrix{
0\ar[r] &\mr{Pic}(X)\ar[r]\ar[d]^{\mr{deg}} &\mr{Pic}(X^{\mr{log}})\ar[r]\ar[d]^{\mr{deg}} &\bigoplus_{x\in S}\Q/\Z\ar[r]\ar[d]^{\mr{sum}} &0 \\
0\ar[r] &\Z\ar[r] &\Q\ar[r] &\Q/\Z\ar[r] &0
}$$
with exact rows.
\item The group $H^2_{\mr{kfl}}(X,\Gm)$ fits into an exact sequence
$$0\rightarrow H^2_{\mr{fl}}(X,\Gm)\rightarrow H^2_{\mr{kfl}}(X,\Gm)\rightarrow \bigoplus_{x\in S}H^1_{\mr{fl}}(x,\Q/\Z) \rightarrow \Q/\Z.$$
If $K$ is a function field with its field of constants algebraically closed, then we have $H^2_{\mr{kfl}}(X,\Gm)\cong H^2_{\mr{fl}}(X,\Gm)=0$. If the residue fields at the points of $S$ are finite, then the above exact sequence becomes
$$0\rightarrow H^2_{\mr{fl}}(X,\Gm)\rightarrow H^2_{\mr{kfl}}(X,\Gm)\rightarrow \bigoplus_{x\in S}\Q/\Z \rightarrow \Q/\Z$$
with $H^2_{\mr{fl}}(X,\Gm)=\begin{cases}
(\Z/2\Z)^{r-1}, &\text{ if $K$ has $r>0$ real places;}   \\
0, &\text{ otherwise.}
\end{cases}$
\end{enumerate}
\end{prop}
\begin{proof}
We are left with checking $H^2_{\mr{fl}}(X,\Gm)=0$ for $X$ a smooth projective curve over an algebraically closed field, and $H^1_{\mr{fl}}(x,\Q/\Z)\cong\Q/\Z$ for $x$ a point with finite residue field. The first follows from \cite[\href{https://stacks.math.columbia.edu/tag/03RM}{Tag 03RM}]{stacks-project}, and the second follows from
$$H^1_{\mr{fl}}(x,\Q/\Z)\cong H^1_{\mr{\acute{e}t}}(x,\Q/\Z)\cong\mr{Hom}(\hat{\Z},\Q/\Z)\cong\Q/\Z.$$
\end{proof}

\begin{rmk}
A homomorphism $(\Q/\Z)^n\rightarrow\Q/\Z$ is either zero or surjective. It follows that, in the case that the residue fields at the points of $S$ are finite, we have either a short exact sequence
$$0\rightarrow H^2_{\mr{fl}}(X,\Gm)\rightarrow H^2_{\mr{kfl}}(X,\Gm)\rightarrow \bigoplus_{x\in S}\Q/\Z \rightarrow 0$$
or an exact sequence
$$0\rightarrow H^2_{\mr{fl}}(X,\Gm)\rightarrow H^2_{\mr{kfl}}(X,\Gm)\rightarrow \bigoplus_{x\in S}\Q/\Z \rightarrow \Q/\Z\rightarrow0.$$
\end{rmk}

\appendix
\section{}
\begin{lem}\label{lemA.1}
Let $S$ be a scheme, $F$ a finite flat commutative group scheme of multiplicative type over $S$ which is killed by some positive integer $n$, and $G$ a commutative group scheme over $S$ satisfying one of the following two conditions
\begin{enumerate}[(1)]
\item $G$ is smooth and affine over $S$; 
\item $G[n]:=\mr{Ker}(G\xrightarrow{n}G)$ is finite flat over $S$. 
\end{enumerate}
Then the fppf sheaf $H:=\mc{H}om_S(F,G)$ is representable by an \'etale quasi-finite separated group scheme over $S$;
\end{lem}
\begin{proof}
We first deal with case (1). By \cite[Exp. XI, Cor. 4.2]{sga3-2}, the sheaf $\mc{H}om_S(F,G)$ is representable by a smooth separated group scheme over $S$. The fibres of the group scheme $\mc{H}om_S(F,G)$ over $S$ are finite by the structure theorem \cite[Exp. XVII, Thm. 7.2.1]{sga3-2} of commutative group schemes and \cite[Exp. XVII, Prop. 2.4]{sga3-2}. It follows that $\mc{H}om_S(F,G)$ is a quasi-finite \'etale separated group scheme over $S$.

Now we deal with case (2). Clearly we have $\mc{H}om_S(F,G)=\mc{H}om_S(F,G[n])$. Let $0\rightarrow G[n]\rightarrow G_1\rightarrow G_2\rightarrow 0$ be the canonical smooth resolution of $G$ with $G_1,G_2$ affine smooth commutative group schemes over $S$, see \cite[Thm. A.5]{mil2}. Then we have an exact sequence
$$0\rightarrow \mc{H}om_S(F,G)\rightarrow \mc{H}om_S(F,G_1)\xrightarrow{\alpha} \mc{H}om_S(F,G_2)$$
of fppf sheaves of abelian groups over $S$. By case (1), the sheaves $\mc{H}om_S(F,G_i)$ for $i=1,2$ are representable by \'etale quasi-finite separated group schemes over $S$. Hence $H=\mc{H}om_S(F,G)$ as the kernel of $\alpha$ is representable. Furthermore by \cite[\href{https://stacks.math.columbia.edu/tag/02GW}{Tag 02GW}]{stacks-project}, $\alpha$ is \'etale. It is also separated. It follows that $H$ is \'etale, separated, and quasi-finite. 
\end{proof}

\section*{Acknowledgement}
The author has benefited a lot from regular discussions on this project with Professor Ulrich G\"ortz in the past one and a half years. This article would never exist without his help. The author thanks Professor Chikara Nakayama for very helpful correspondences, as well as for telling him the reference \cite{swa1} for cup-products on an arbitrary site. The author thanks the anonymous referee for corrections and very helpful suggestions. This work has been partially supported by SFB/TR 45 ``Periods, moduli spaces and arithmetic of algebraic varieties''.
\bibliographystyle{alpha}
\bibliography{bib}

\end{document}